\documentclass[twoside]{amsart}
\usepackage{amssymb}
\usepackage{bm}
\usepackage{tikz-cd}
\usetikzlibrary{calc}
\usepackage{url}
\usepackage{etoolbox}
\usepackage{hyperref}
\usepackage[greek,english]{babel}
\usepackage{mathtools}
\usepackage[scr=rsfso]{mathalpha}
\usepackage{mathdots}
\usepackage{microtype}

\renewcommand{\digamma}{{\textit{\bfseries\foreignlanguage{greek}{ϝ}}}}

\newtheorem{thm}{Theorem}[section]
\newtheorem{lemma}[thm]{Lemma}

\newtheorem{prop}[thm]{Proposition}
\newtheorem{introthm}{Theorem}
\newtheorem*{prop*}{Proposition}

\theoremstyle{remark}
\newtheorem{rk}[thm]{Remark}
\newtheorem{constr}[thm]{Construction}

\newtheorem{convention}[thm]{Convention}

\newtheorem*{claim*}{Claim}
\AtBeginEnvironment{claim*}{\let\oldqed=\qedsymbol%
\renewcommand\qedsymbol{$\triangle$}}
\AtEndEnvironment{claim*}{\let\qedsymbol=\oldqed}

\theoremstyle{definition}
\newtheorem{defi}[thm]{Definition}

\numberwithin{equation}{section}

\newcommand{\nerve}{\textup{N}}
\newcommand{\h}{\textup{h}}
\newcommand{\cat}[1]{\textbf{\textup{#1}}}
\newcommand{\op}{{\textup{op}}}
\newcommand{\blank}{{\textup{--}}}
\newcommand{\pr}{{\textup{pr}}}
\newcommand{\Hom}{{\textup{Hom}}}
\newcommand{\id}{{\textup{id}}}
\newcommand{\ev}{{\textup{ev}}}

\newcommand{\colim}{\mathop{\textup{colim}}\nolimits}
\newcommand{\Fun}{\textup{Fun}}

\newcommand{\maps}{\mathord{\textup{maps}}}
\newcommand{\Ob}{\mathop{\textup{Ob}}}

\newcommand{\core}{\mathop{\textup{core}}}

\newcommand{\ppo}{\mathbin{\raise.25pt\hbox{\scalebox{.67}{$\square$}}}}

\newcommand{\diag}{{\textup{diag}}}

\newcommand{\Ho}{{\textup{Ho}}}

\newcommand{\Agl}{A^{\textup{gl}}}
\newcommand{\Aut}{\textup{Aut}}
\newcommand{\tni}{\mathord{\reflectbox{$\int$}}}
\newcommand{\F}{{\mathbb F}}

\newcommand{\AglGamma}{{\mathbb A}_{\textup{gl}}}

\newcommand{\Exc}{{\textup{Exc}}}
\newcommand{\Sp}{{\mathscr S\!p}}
\newcommand{\Cc}{{\mathcal C}}
\newcommand{\Dd}{{\mathcal D}}
\renewcommand{\int}{{\smallint}}

\let\del=\partial
\let\smashp=\wedge

\let\phi=\varphi

\def\twocell[#1]{\arrow[#1, dash, phantom, "\Rightarrow"{scale=1.125, yshift=-.4pt, description, allow upside down, sloped, inner sep=0pt}]}

\title[Global homotopy theory via spectral Mackey functors]{Global homotopy theory\\ via spectral Mackey functors}
\author{Tobias Lenz}
\address{\null\kern-\parindent Mathematisches Institut, Rheinische Friedrich-Wilhelms-Universit\"at Bonn, Endenicher Allee 60, 53115 Bonn, Germany}

\begin{document}
\begin{abstract}
We show that Hausmann's model of global stable homotopy theory in terms of symmetric spectra is equivalent to the $\infty$-category of spectral Mackey functors in the sense of Barwick on a certain \emph{global effective Burnside category}. We moreover provide an analogous description of Schwede's ultra-commutative monoids as space-valued global Mackey functors.
\end{abstract}

\maketitle

\section*{Introduction}
One of the defining features of \emph{equivariant homotopy theory} that distinguishes it from the na\"ive homotopy theory of $G$-objects in spaces or spectra is the existence of \emph{genuine fixed points}, which are typically different from the usual homotopy fixed points. In the unstable setting, \emph{Elmendorf's Theorem} \cite{elmendorf} makes precise in which sense a $G$-space can be understood in terms of its fixed point spaces and the restriction and conjugation maps between them: for every group $G$, sending a $G$-space to the collection of its fixed point spaces together with the above structure maps provides an equivalence between unstable $G$-equivariant homotopy theory and the $\infty$-category of presheaves on the category of finite transitive $G$-sets.

Stably, the theory becomes more complicated: in particular, for any finite group $G$, the fixed points of a \emph{genuine $G$-spectrum} come with covariant transfers along subgroup inclusions in addition to the above contravariant functoriality; on the level of homotopy groups the interaction between these two directions of functoriality is governed by the so-called \emph{Mackey double coset formula}, giving the homotopy groups the structure of a \emph{$G$-Mackey functor}. It is then natural to ask whether we can also understand genuine $G$-spectra in terms of their fixed points together with these restriction and transfer maps. However, while the Schwede-Shipley Theorem \cite{schwede-shipley} provides a model of genuine stable $G$-equivariant homotopy theory in terms of spectral presheaves for purely abstract reasons, the resulting indexing category is a priori very far from having an algebraic or combinatorial description, and in particular it is not clear in which sense it encodes restrictions, transfers, and a `coherent double coset formula' between them.

Because of this, Guillou and May \cite{guillou-may-ps} constructed an alternative (spectrally enriched) indexing category built from a $2$-category of spans of finite $G$-sets by `local higher group completion,' and provided an equivalence between spectral presheaves on this indexing category and genuine stable $G$-equivariant homotopy theory, which we can view as a stable analogue of Elmendorf's Theorem. Motivated by this, Barwick \cite{barwick-mackey} developed a general theory of \emph{spectral Mackey functors} formalizing the idea of functors with both covariant and contravariant functoriality in a suitable base ($\infty$-)category $\mathscr I$, together with higher homotopies encoding a Mackey double coset formula between these two directions. Specializing Barwick's theory to the ordinary category of finite $G$-sets yields a theory of \emph{spectral $G$-Mackey functors}; a proof that this is equivalent to classical genuine stable $G$-equivariant homotopy theory has been sketched by Nardin \cite{nardin-thesis}, and a full proof (along different lines) appears in work of Clausen, Mathew, Naumann, and Noel \cite{clausen-mathew-naumann-noel}.

In this paper, we are concerned with the analogous story for \emph{global homotopy theory} in the sense of \cite{schwede-book,hausmann-global}. Roughly and intuitively speaking, global objects are compatible families of genuine $G$-equivariant objects for all finite groups $G$, and as a concrete manifestation of this slogan they come with genuine fixed points for all finite groups $G$. In the unstable case, work of Körschgen \cite{koerschgen} and Schwede \cite{schwede-orbi} provides an analogue of Elmendorf's Theorem, identifying global spaces with presheaves on the $2$-category of finite groups, all homomorphisms, and conjugations.\footnote{In fact, both Schwede and Körschgen work in the more general context of global homotopy theory with respect to compact Lie groups.} Similarly to the $G$-equivariant setting, the corresponding stable theory is richer: in particular, global spectra again admit covariant transfers along injective homomorphisms in addition to the contravariant functoriality inherited from the unstable world, and on the level of homotopy groups there is an analogue of the Mackey double coset formula relating the two, making the homotopy groups into so-called \emph{global (Mackey) functors}.

It is then again natural to ask whether we can understand global stable homotopy theory via Barwick's framework, and indeed there is a natural candidate for a theory of \emph{global spectral Mackey functors}. However, while several sources \cite{parametrized-intro, berman-thesis} mention the existence of an equivalence between global spectral Mackey functors and global spectra, so far no proof of this had appeared in the literature and the equivariant proof from \cite{clausen-mathew-naumann-noel} can not be immediately adapted to the global setting. The purpose of this paper is to end this unpleasant state of affairs: we will explain how one can apply Barwick's machinery to the $2$-category of finite groupoids to obtain a notion of \emph{global spectral Mackey functors}, and we prove:

\begin{introthm}[{Theorem~\ref{thm:main-thm}}]\label{introthm:mackey-vs-spectra}
There exists an equivalence between the $\infty$-category of global spectral Mackey functors and the $\infty$-category of global spectra.
\end{introthm}

We will actually deduce this theorem from a `non-group-completed' comparison, describing Schwede's \emph{ultra-commutative monoids} \cite{schwede-book} (a `genuine' global version of $E_\infty$-monoids, again supporting the appropriate notion of transfers) in terms of \emph{space-valued} Mackey functors:

\begin{introthm}[{Theorem~\ref{thm:ucom}}]\label{introthm:ucom}
    There exists an equivalence between the $\infty$-category of global space-valued Mackey functors and the $\infty$-category of ultra-commutative monoids.
\end{introthm}

\subsection*{Historical remark}
The first version of this paper appeared on the arXiv in early 2022 under the name `Global algebraic $K$-theory is Swan $K$-theory.' As the title indicates, this version contained another main result, describing the image of Schwede's \emph{global algebraic $K$-theory spectra} \cite{schwede-k-theory} under the above equivalence; a generalization of this result will instead appear as a separate paper \cite{swan-two}.

In the meantime, Theorems~\ref{introthm:mackey-vs-spectra} and \ref{introthm:ucom} above have been reproven and generalized in joint work of Cnossen, Linskens, and myself \cite{CLL_Global, CLL_Spans}. This alternative proof proceeds by establishing a universal property for global spectra (and more generally $G$-global spectra in the sense of \cite{g-global}) in the setting of \emph{parametrized higher category theory} \cite{parametrized-intro} and then showing that ($G$-)global spectral Mackey functors possess the same universal property. While this is a more conceptual approach (and yields a stronger result), the proof is quite involved, and the above two articles together amount to almost 200 pages. I therefore think there is value in having a reasonably short elementary proof of Theorems~\ref{introthm:mackey-vs-spectra} and~\ref{introthm:ucom}, in particular since the techniques of the present paper are general enough to apply to other contexts. For example, since the first version of this paper appeared on the arXiv, Marc has used similar ideas to establish Mackey functor models of $N_\infty$-monoids \cite{marc-Noo}.

\subsection*{Outline}
In Section~\ref{sec:gamma} we recall our reference model of `global commutative monoids,' and we use a Morita type argument to show that it can be equivalently described as product-preserving presheaves on a certain (mysterious) $\infty$-category $\AglGamma$. Section~\ref{sec:burnside} then constructs the global effective Burnside category $A^\text{gl}$ using Barwick's machinery, and does most of the heavy lifting of the present paper by showing that $\AglGamma$ is equivalent to the opposite of $A^\text{gl}$, yielding the proof of Theorem~\ref{introthm:ucom}. Building on this and the relation between global spectra and ultra-commutative monoids established in \cite{g-global}, we then prove Theorem~\ref{introthm:mackey-vs-spectra} in Section~\ref{sec:spectra}.

\subsection*{Acknowledgements}
I would like to thank Bastiaan Cnossen and Sil Linskens for helpful conversations and Stefan Schwede for useful comments on a very early version of this article.

I am moreover grateful to the Isaac Newton Institute for Mathematical Sciences for support and hospitality during the programme `Homotopy harnessing higher structures' in summer 2018, when work on this article began. The first version of this paper was completed while I was in residence at Institut Mittag-Leffler in Djursholm, Sweden, in early 2022 as a participant of the program `Higher algebraic structures in algebra, topology and geometry,' supported by the Swedish Research Council under under grant no.~2016-06596. Finally, I would like to thank the Max Planck Institute for Mathematics in Bonn for their hospitality and support during all the years in between.

Currently, the author is an associate member of the Hausdorff Center for Mathematics
at the University of Bonn (DFG GZ 2047/1, project ID 390685813).

\section{Global special \texorpdfstring{$\Gamma$}{Γ}-spaces}\label{sec:gamma}
In this section, we will introduce the model category of \emph{global $\varGamma$-spaces} and its full subcategory of \emph{special global $\varGamma$-spaces}, whose objects should be viewed as encoding \emph{genuinely commutative global monoids}. As the main result of this section, we will show that special global $\Gamma$-spaces can be equivalently described as product preserving presheaves on a certain (mysterious) $\infty$-category $\AglGamma$.

\subsection{A reminder on global spaces} We begin by recalling our reference model of \emph{unstable global homotopy theory} \cite[Chapter~1]{schwede-book}, and more generally \emph{unstable $G$-global homotopy theory} in the sense of \cite[Chapter~1]{g-global}. Crucial to this is a certain simplicial monoid that we call the \emph{universal finite group}:

\begin{defi}\label{def:universe}
Let $\mathcal M$ be the monoid of self-injections of the countably infinite set $\omega=\{0,1,\dots\}$. We let $E$ denote the right adjoint of the functor $\cat{SSet}\to\cat{Set}$ sending a simplicial set to its set of vertices, and we call the simplicial set $E\mathcal M$ with the induced monoid structure the \emph{universal finite group}.
\end{defi}

Explicitly, $(E\mathcal M)_n=\mathcal M^{1+n}$ with the obvious functoriality and with pointwise multiplication. Moreover, $E\mathcal M$ is canonically isomorphic to the nerve of the indiscrete category with object set $\mathcal M$, which we will again denote by $E\mathcal M$.

\begin{defi}
We call a finite subgroup $H\subset\mathcal M$ \emph{universal} if $\omega$ with $H$ acting via the restriction of the tautological $\mathcal M$-action is a \emph{complete $H$-set universe}, i.e.~every finite (or, equivalently, every countable) $H$-set embeds into $\omega$ equivariantly.
\end{defi}

\begin{thm}\label{thm:g-global-model-structure-EM}
Let $G$ be any discrete group. Then the category $\cat{$\bm{E\mathcal M}$-$\bm G$-SSet}$ of simplicial sets with an action of the simplicial monoid $E\mathcal M\times G$ admits a unique model structure in which a map $f$ is a weak equivalence or fibration if and only if for all universal $H\subset\mathcal M$ and all $\phi\colon H\to G$ the map $f^\phi$ is a weak homotopy equivalence or Kan fibration, respectively; here we write $(\blank)^\phi$ for the fixed points with respect to the \emph{graph subgroup} $\Gamma_{H,\phi}\coloneqq\{(h,\phi(h)):h\in H\}\subset\mathcal M\times G$.

We call this model structure the \emph{$G$-global model structure} and its weak equivalences the \emph{$G$-global weak equivalences}. It is simplicial, proper, combinatorial with generating cofibrations
\begin{equation*}
\{E\mathcal M\times_\phi G\times(\del\Delta^n\hookrightarrow\Delta^n) : \text{$H\subset\mathcal M$ universal, $\phi\colon H\to G$}\},
\end{equation*}
and filtered colimits and finite products in it are homotopical.
\begin{proof}
See \cite[Corollary~1.2.34 and Lemma~1.1.3]{g-global}.
\end{proof}
\end{thm}

\begin{rk}
    It is not hard to show that every abstract finite group admits an essentially unique embedding $H\hookrightarrow\mathcal M$ such that the image is a universal subgroup, see~\cite[Lemma~1.2.8]{g-global}. Thus, taking $G=1$ this means that a global space comes with genuine $H$-fixed points for every finite group $H$. More generally, a $G$-global space comes with fixed points for every homomorphism $\phi\colon H\to G$. While we will not need this below, we remark for motivational purposes that the $G$-global Elmendorf theorem allows us to equivalently describe $G$-global spaces in terms of these fixed-point spaces together with various `restriction' maps between them \cite[Theorem~1.1.12]{g-global}.
\end{rk}

\begin{rk}
We again specialize the above theorem to $G=1$. In \cite[Theorem~1.2.21]{schwede-book} Schwede introduced a \emph{global model structure} on the category of orthogonal spaces  which contains equivariant information for all compact Lie groups. After throwing away all information for infinite groups (i.e.~after Bousfield localizing at a certain explicit class of `$\mathcal Fin$-global weak equivalences'), this becomes Quillen equivalent to the above model category, see~\cite[Theorems~1.4.30 and 1.4.31, Corollary~1.5.29]{g-global}.
\end{rk}

While the existence of suitable model structures will be important for some arguments below, ultimately we are interested in statements about the $\infty$-categories they present. More precisely, whenever $\mathscr C$ is a category equipped with a wide subcategory $W$ of weak equivalences, we can form the \emph{Dwyer-Kan localization} of $\mathscr C$ at $W$, i.e.~the universal example of a functor $\gamma\colon\nerve\mathscr C\to\mathscr C^\infty_W$ into some $\infty$-category sending $W$ into the maximal Kan complex $\core(\mathscr C^\infty_W)$; one way to obtain this is to form the \emph{Hammock localization} of $(\mathscr C,W)$ and then to apply the right derived functor of the homotopy coherent nerve $\nerve_\Delta$, see e.g.~\cite[1.2]{dwyer-kan-revisited}. If $W$ is clear from the context, we will just denote any localization by $\mathscr C^\infty$.

\begin{convention}
In order to avoid cumbersome notation, we agree to pick all localizations to be the respective identities on objects (which is for example possible since for the Hammock localization and the usual fibrant replacement functor in simplicial categories the resulting functor is bijective on objects), and we won't notationally distinguish between a morphism in the $1$-category and its image under the localization.
\end{convention}

We close this subsection with a technical observation: as a simplicial model category, $\cat{$\bm{E\mathcal M}$-$\bm G$-SSet}$ comes with a geometric realization functor
\begin{equation}\label{eq:geom-real-EM}
\Fun(\Delta^\op,\cat{$\bm{E\mathcal M}$-$\bm G$-SSet})\to \cat{$\bm{E\mathcal M}$-$\bm G$-SSet}
\end{equation}
given as the coend $\int^{[n]\in\Delta^\op}\Delta^n\otimes (\blank)_n$. While this functor is always left Quillen for the projective or Reedy model structure on the source and hence preserves weak equivalences between appropriately cofibrant objects, in our case we even have:

\begin{lemma}\label{lemma:EM-geometric-realization-homotopical}
The geometric realization functor $(\ref{eq:geom-real-EM})$
preserves all $G$-global weak equivalences.
\begin{proof}
As $\cat{$\bm{E\mathcal M}$-$\bm G$-SSet}$ is an enriched functor category, geometric realizations can be computed in $\cat{SSet}$, so they are simply given by taking the diagonal. The claim therefore follows from \cite[Lemma~1.2.57]{g-global}.
\end{proof}
\end{lemma}

\subsection{A reminder on global \texorpdfstring{$\bm\Gamma$}{Γ}-spaces} Recall that Segal's \emph{$\varGamma$-spaces} \cite{segal-gamma} provide a rigorous implementation of the idea of monoids which are commutative, associative, and unital up to a system of `coherent' homotopies. Similarly to the equivariant situation, the correct notion of `commutative monoids up to globally coherent homotopies' turns out to be more subtle than just commutative monoids in the $\infty$-category of global spaces, a phenomenon for which Schwede coined the term \emph{ultra-commutativity}. Morally, this subtlety of (global) equivariant commutativity comes from the fact that we can let a group act on an iterated product $\prod_{i=1}^nx$ by permuting the factors; in the homotopy coherent setting, the commutativity of the product is extra \emph{data}, and we have to keep track how this relates to the above permutation action. In our model this is accomplished by looking at the individual terms of a $\Gamma$-object in $\cat{$\bm{E\mathcal M}$-SSet}$ not just through the eyes of the global weak equivalences, but instead through the eyes of the $\Sigma_n$-global weak equivalences for varying $n\ge0$:

\begin{defi}
    We write $\Gamma$ for the category of finite pointed sets and let $n^+\in\Gamma$ denote the object $\{0,1,\dots,n\}$ with basepoint $0$.
    A \emph{global $\varGamma$-space} is a functor $\Gamma\to\cat{$\bm{E\mathcal M}$-SSet}$ sending $0^+$ to the terminal object.
\end{defi}

\begin{thm}\label{thm:Gamma-EM-level}
There is a unique model structure on the category $\cat{$\bm\Gamma$-$\bm{E\mathcal M}$-SSet}_*$ of global $\varGamma$-spaces in which a map $f\colon X\to Y$ is a weak equivalence or fibration if and only if $f(n^+)$ is a $\Sigma_n$-global weak equivalence for all $n\ge0$; here we equip $X(n^+)$ and $Y(n^+)$ with the $\Sigma_n$-action induced by the functoriality of $X$ in $\Gamma$.

We call this model structure the \emph{global level model structure}. It is simplicial and combinatorial, and filtered colimits in it are homotopical.
\begin{proof}
See \cite[Theorem~2.2.24]{g-global}.
\end{proof}
\end{thm}

Just like in the classical non-equivariant situation, we will mostly be interested in those global $\Gamma$-spaces that satsify an additional \emph{specialness condition}. Again, the correct notion of specialness is more refined than the na\"ive generalization of non-equivariant specialness and has to take the $\Sigma_n$-actions into account.

\begin{defi}
A global $\Gamma$-space is called \emph{special} if the usual Segal map $X(n^+)\to X(1^+)^{\times n}$ is a $\Sigma_n$-global weak equivalence for every $n\ge0$. Here $\Sigma_n$ acts on the left as before and on the right via permuting the factors. We write $\cat{$\bm\Gamma$-$\bm{E\mathcal M}$-SSet}_*^{\textup{special}}$ for the full subcategory spanned by the special global $\Gamma$-spaces.
\end{defi}

By \cite[Corollary~2.2.53 and and Theorem~2.2.55]{g-global}, the resulting $\infty$-category $(\cat{$\bm\Gamma$-$\bm{E\mathcal M}$-SSet}_*^{\textup{special}})^\infty$ is presentable.

\begin{rk}
Schwede originally introduced so-called \emph{ultra-commutative monoids} \cite[Chapter~2]{schwede-book} as his model for global coherent commutativity. While these have the advantage of containing meaningful information for all compact Lie groups, as long as one is only interested in global homotopy theory for finite groups (as we are in this paper), there is no difference to the homotopy theory of global special $\Gamma$-spaces, see \cite[Corollary~2.3.17]{g-global}.
\end{rk}

\subsection{Special global \texorpdfstring{$\bm\Gamma$}{Γ}-spaces as product-preserving presheaves} As a first step towards our $\infty$-categorical model of special global $\Gamma$-spaces, we will show that $(\cat{$\bm\Gamma$-$\bm{E\mathcal M}$-SSet}_*^\text{special})^\infty$ is equivalent to product-preserving presheaves on a certain full subcategory $\AglGamma\subset(\cat{$\bm\Gamma$-$\bm{E\mathcal M}$-SSet}_*^\text{special})^\infty$. The bulk of the work that goes into proving Theorem~\ref{introthm:ucom} will then be actually understanding this full subcategory, to which all of Section~\ref{sec:burnside} will be devoted.

We begin by describing the objects of the full subcategory $\AglGamma$. These will arise from certain permutative categories (i.e.~symmetric monoidal 1-categories in which the associativity and unitality isomorphisms are the respective identities) via a classical construction from the $K$-theory of symmetric monoidal categories:

\begin{constr}\label{constr:segal-may-shimada-shimakawa}
We refer the reader to \cite[Definition~2.1]{shimada-shimakawa} for the construction of the special $\Gamma$-category $\Gamma(\mathcal C)$ associated to a small symmetric monoidal category $\mathcal C$; all that we will need below is that this construction is functorial in the $1$-category $\cat{SymMonCat}^0$ of small symmetric monoidal categories and \textit{strictly unital} strong symmetric monoidal functors, and that its underlying category is naturally isomorphic to $\mathcal C$; for simplicity we will suppress this natural isomorphism below and pretend that $\ev_{1^+}\circ\Gamma$ is equal to the forgetful functor $\cat{SymMonCat}^0\to\cat{Cat}$.

Using this, we now simply define
\begin{equation*}
\Gamma_{\textup{gl}}\coloneqq\nerve\circ\Fun(E\mathcal M,\blank)\circ\Gamma\colon\cat{SymMonCat}^0\to\cat{$\bm\Gamma$-$\bm{E\mathcal M}$-SSet}_*,
\end{equation*}
where $E\mathcal M$ acts on itself from the right in the obvious way, inducing a left $E\mathcal M$-action on $\Fun(E\mathcal M,\blank)$.

Just like $\nerve\circ\Gamma$ lands in special $\Gamma$-spaces, $\Gamma_{\text{gl}}$ factors through the subcategory of special global $\Gamma$-spaces, see~\cite[Example~2.2.52]{g-global}.
\end{constr}

\begin{rk}\label{rk:Gamma-gl-extension}
The fact that $\Gamma$ (and hence $\Gamma_{\text{gl}}$) is only functorial in \emph{strictly unital} symmetric monoidal functors is just a minor annoyance, and classical strictification theory allows us to extend it formally to the whole $(2,1)$-category of (small) symmetric monoidal categories and strong symmetric monoidal functors as follows:

The inclusion  $\cat{SymMonCat}^0\hookrightarrow\cat{SymMonCat}$ into the $1$-category of symmetric monoidal categories and all symmetric monoidal functors is a homotopy equivalence, see e.g.~\cite[Proposition~6.7]{sym-mon-global}. Moreover, the inclusion of the $1$-category $\cat{SymMonCat}$ of small symmetric monoidal categories into the corresponding $(2,1)$-category $\cat{SymMonCat}_{(2,1)}$ is a simplicial localization by a standard argument, cf.~\cite[Proposition~A.1.10]{g-global}, and so the composite
\begin{equation*}
\nerve(\cat{SymMonCat}^0)\xrightarrow{\nerve(\Gamma_{\textup{gl}})}\nerve(\cat{$\bm\Gamma$-$\bm{E\mathcal M}$-SSet}_*^{\textup{special}})\xrightarrow{\textup{loc}}(\cat{$\bm\Gamma$-$\bm{E\mathcal M}$-SSet}_*^{\textup{special}})^\infty
\end{equation*}
factors through a functor $\nerve_\Delta(\cat{SymMonCat}_{(2,1)})\to(\cat{$\bm\Gamma$-$\bm{E\mathcal M}$-SSet}_*^{\textup{special}})^\infty$ in an essentially unique way; we pick one such factorization and denote it by $\Gamma_{\text{gl}}$ again.
\end{rk}

\begin{constr}[{cf.\ \cite[Construction~4.2.15]{g-global}}]\label{constr:mathfrak-FG}
    Let $G$ be a finite group. We write $\mathfrak F_G$ for the following permutative category: the objects of $\mathfrak F_G$ are given by the natural numbers $\textbf{0},\textbf{1},\dots$, while the hom-sets are given by $\hom(\bm{m},\bm{n})=\emptyset$ for $m\not=n$ and $\hom(\bm{n},\bm{n})=\Sigma_n\wr G$. The symmetric monoidal structure is given on objects by the addition on $\mathbb N$, and on morphisms by block sum and concatenation, i.e.~$(\sigma;g_1,\dots,g_m)\oplus(\tau;h_1,\dots,h_n)\coloneqq (\sigma\oplus\tau;g_1,\dots,g_m,h_1,\dots,h_n)$. The symmetry isomorphism $\bm{m+n}\to\bm{n+m}$ is given by $(\chi_{m,n};1,\dots,1)$ where $\chi_{m,n}\in\Sigma_{m+n}$ is the shuffle permutation moving the first $m$ entries to the end.
\end{constr}

In other words, $\mathfrak F_G$ is a specific choice of a skeleton of the groupoid core of the category of finite free $G$-sets.

\begin{defi}
    We write $\AglGamma\subset (\cat{$\bm\Gamma$-$\bm{E\mathcal M}$-SSet}_*^\text{special})^\infty$ for the full subcategory spanned by finite products of the objects of the form $\Gamma_\text{gl}(\mathfrak F_G)$ for all finite groups $G$.
\end{defi}

In order to identify $(\cat{$\bm\Gamma$-$\bm{E\mathcal M}$-SSet}_*^\text{special})^\infty$ with product-preserving presheaves on $\AglGamma$, it will suffice by the general theory of \cite[§5.5.8]{htt} to show that the objects $\Gamma_\text{gl}(\mathfrak F_G)$ form a set of \emph{compact projective generators}, i.e.~the hom functors $\hom(\Gamma_\text{gl}(\mathfrak F_G),-)$ are jointly conservative and preserve filtered colimits as well as $\Delta^\op$-indexed colimits. We therefore begin by making these corepresented functors explicit:

\begin{thm}\label{thm:fixed-points-corep}
Let $G\subset\mathcal M$ be universal. The functor
\begin{equation*}
(\cat{$\bm\Gamma$-$\bm{E\mathcal M}$-SSet}_*^{\textup{special}})^\infty\to\cat{SSet}^\infty\eqqcolon\mathscr S
\end{equation*}
induced by $X\mapsto X(1^+)^G$ is corepresented by $\Gamma_{\textup{gl}}(\mathfrak F_G)$ with universal class
\begin{equation*}
\tau\in\pi_0\Gamma_{\textup{gl}}(\mathfrak F_G)(1^+)=\pi_0(\Fun(E\mathcal M,\mathfrak F_G)^G)
\end{equation*}
given by the isomorphism class of some (hence any) $G$-fixed functor $T\colon E\mathcal M\to\mathfrak F_G$ such that $T(g_2,g_1)=(\id;g_2g_1^{-1})\colon\cat{1}\to\cat{1}$ for all $g_1,g_2\in G\subset\mathcal M$.
\begin{proof}
We first note that $\tau$ is well-defined: namely, the functor $EG\to\mathfrak F_G$ specified above is a $G$-fixed object of $\Fun(EG,\mathfrak F_G)$ by direct inspection, and as $EG\hookrightarrow E\mathcal M$ is a right $G$-equivariant equivalence of categories (i.e.~an equivalence in the $2$-category of right $G$-categories, right $G$-equivariant functors, and right $G$-equivariant natural transformations), it follows that this admits an essentially unique extension to an object $T\in\Fun(E\mathcal M,\mathfrak F_G)^G$.

As $\Gamma(1^+,\blank)\smashp E\mathcal M/G_+$ corepresents the functor $\cat{$\bm\Gamma$-$\bm{E\mathcal M}$-SSet}_*^\infty\to\mathscr S$ sending a global $\Gamma$-space $X$ to $X(1^+)^G$  (with universal class the component of $\id_{1^+}\smashp[1])$, it then suffices to observe that there is a special weak equivalence\footnote{Here we call a map $f$ in $\cat{$\bm\Gamma$-$\bm{E\mathcal M}$-SSet}_*^\infty$ a \emph{special weak equivalence} if the induced map $\maps(f,T)$ on mapping spaces is an equivalence for every special global $\Gamma$-space $T$.} $\Gamma(1^+,\blank)\smashp E\mathcal M/G_+\to\Gamma_{\text{gl}}(\mathfrak F_G)$ sending $\id_{1^+}\smashp[1]$ to $T$ by \cite[Theorem~4.2.22]{g-global}.
\end{proof}
\end{thm}

\begin{prop}\label{prop:fixed-points-cocontinuous}
Let $G\subset\mathcal M$ be universal. Then the functor
\begin{equation*}
(\blank)(1^+)^G\colon(\cat{$\bm\Gamma$-$\bm{E\mathcal M}$-SSet}_*^\textup{special})^\infty\to\mathscr S
\end{equation*}
preserves filtered colimits and $\Delta^\op$-shaped colimits.
\begin{proof}
We will first prove the corresponding statements for $(\cat{$\bm\Gamma$-$\bm{E\mathcal M}$-SSet}_*)^\infty$. For this, note that filtered colimits in the $1$-category of global $\Gamma$-spaces are homotopical by Theorem~\ref{thm:Gamma-EM-level}, and so are geometric realization by Lemma~\ref{lemma:EM-geometric-realization-homotopical} applied levelwise. It follows directly that filtered colimits in the $\infty$-category of global $\Gamma$-spaces can be computed in the underlying model category, while \cite[Corollary~A.2.9.30]{htt} shows that $\Delta^\op$-shaped colimits can be computed as geometric realizations. However, taking $G$-fixed points clearly preserves filtered colimits on the pointset level, and it also commutes with geometric realization as the latter is just given by taking diagonals.

It only remains to show that $(\cat{$\bm\Gamma$-$\bm{E\mathcal M}$-SSet}_*^\textup{special})^\infty$ is closed under the above colimits, for which it is then again enough to show the corresponding statements on the pointset level. This follows simply from Theorem~\ref{thm:g-global-model-structure-EM} and Lemma~\ref{lemma:EM-geometric-realization-homotopical} using that finite products in $\cat{SSet}$ commute with filtered colimits and geometric realization.
\end{proof}
\end{prop}

\begin{prop}\label{prop:gamma-mackey-geometric}
    The restricted Yoneda embedding induces an equivalence
    \begin{equation}\label{eq:restricted-Yoneda-CMON}
        (\cat{$\bm\Gamma$-$\bm{E\mathcal M}$-SSet}_*^\textup{special})^\infty\simeq\Fun^\times(\AglGamma^\op,\mathscr S).
    \end{equation}
\begin{proof}
    Recall that the Yoneda embedding $\AglGamma\hookrightarrow\Fun^\times(\AglGamma^\op,\mathscr S)$ exhibits the target as the sifted colimit completion of the source, so that there exists a unique sifted colimit-preserving functor $L\colon\Fun^\times(\AglGamma^\op,\mathscr S)\to (\cat{$\bm\Gamma$-$\bm{E\mathcal M}$-SSet}_*^\textup{special})^\infty$ extending the inclusion of $\AglGamma$; moreover, $L$ is left adjoint to $(\ref{eq:restricted-Yoneda-CMON})$.

    Combining Theorem~\ref{thm:fixed-points-corep} with Proposition~\ref{prop:fixed-points-cocontinuous}, $\Gamma_\text{gl}(\mathfrak F_G)$ is compact projective for every finite group $G$, hence so is any finite coproduct of these. As finite coproducts and products agree in $(\cat{$\bm\Gamma$-$\bm{E\mathcal M}$-SSet}_*^\textup{special})^\infty$ by \cite[Theorem~2.2.64]{g-global}, we conclude that $\AglGamma$ consists of compact projective objects. \cite[Proposition~5.5.8.22]{htt} therefore shows that $L$ is fully faithful, and so it only remains to prove that its right adjoint $(\ref{eq:restricted-Yoneda-CMON})$ is conservative.

    In light of Theorem~\ref{thm:fixed-points-corep} it will be enough for this that a map $f\colon X\to Y$ of global {special} $\Gamma$-spaces is a global level weak equivalence as soon as $f(1^+)$ is a global weak equivalence. As $f(n^+)$ for $n\ge2$ agrees up to $\Sigma_n$-global weak equivalence with $\prod_{i=1}^n f(1^+)$, it will suffice that $\prod_{i=1}^n(-)$ sends global weak equivalences to $\Sigma_n$-global weak equivalences. Identifying $\{1,\dots,n\}$ with $\Sigma_n/\Sigma_{n-1}$, this follows from \cite[Corollary~1.2.79]{g-global}, also cf.\ \cite[Corollary~1.4.71]{g-global}.
\end{proof}
\end{prop}

\section{The global effective Burnside category}\label{sec:burnside}
As $\AglGamma$ was defined as a full subcategory of $(\cat{$\bm\Gamma$-$\bm{E\mathcal M}$-SSet}_*^\textup{special})^\infty$, it is still a rather mysterious object. In this section, we will give a concrete and combinatorial description as the opposite of a certain \emph{global effective Burnside category} $A^\text{gl}$ (Definition~\ref{defi:burnside}). As part of this, we will in particular show that $\AglGamma$ is actually a $(2,1)$-category, i.e.~its mapping spaces are 1-truncated.

\subsection{The global effective Burnside category} In this subsection we will construct $A^\text{gl}$ using Barwick's general machinery of effective Burnside categories \cite{barwick-mackey}.

\begin{defi}
We write $\mathscr F$ for the $\infty$-category of finite groupoids (i.e.~the homotopy coherent or Duskin nerve of the $(2,1)$-category of finite groupoids) and $\mathscr F_\dagger\subset\mathscr F$ for the wide subcategory of faithful functors.
\end{defi}

The \emph{canonical model structure} on the category $\cat{Cat}$ of small categories (see e.g.~\cite{rezk-cat}) restricts to a model structure on $\cat{Grpd}$. Explicitly, the weak equivalences are the equivalences of groupoids, the cofibrations are the functors that are injective \emph{on objects}, and the fibrations are given by the \emph{isofibrations}, i.e.~those functors $F\colon\mathcal G\to\mathcal H$ such that there exists for every $G\in\mathcal G$ and every $h\colon F(G)\to H$ in $\mathcal H$ a morphism $g\colon G\to G'$ in $\mathcal G$ with $F(g)=h$. This model structure is combinatorial and it is moreover simplicial with respect to the obvious enrichment. We will frequently use below that we can use this structure to compute colimits and limits in $\mathscr F$:

\begin{lemma}\label{lemma:limits-vs-homotopy-limits}
\begin{enumerate}
\item $\mathscr F$ has all finite coproducts and these can be computed in $\cat{Grpd}$, i.e.~the localization functor preserves finite coproducts.
\item $\mathscr F$ has all pullbacks and they can be computed as homotopy pullbacks in $\cat{Grpd}$ in the following sense: every functor $\mathcal G\to\mathcal D$ in $\mathscr F$ can be represented by an equivalence followed by an isofibration $\mathcal C\to\mathcal D$ of \emph{finite} groupoids, and any $1$-categorical pullback
\begin{equation*}
\begin{tikzcd}
\mathcal A\arrow[d]\arrow[dr,phantom,"\lrcorner"{very near start}]\arrow[r]&\mathcal B\arrow[d]\\
\mathcal C\arrow[r]&\mathcal D
\end{tikzcd}
\end{equation*}
of finite groupoids along an isofibration $\mathcal C\to\mathcal D$ defines a pullback in $\mathscr F$.
\end{enumerate}
\begin{proof}
We will prove the second statement, the first one being similar but easier.

For this we first note that the $\infty$-category $\overline{\mathscr F}$ of \emph{all} groupoids (i.e.~the homotopy coherent nerve $\nerve_\Delta(\cat{Grpd})$ of the simplicial category $\cat{Grpd}$ of groupoids) is complete and cocomplete and that (co)limits in it can be computed as homotopy (co)limits in $\cat{Grpd}$, see \cite[proof of Corollary~4.2.4.8]{htt}. As $\mathscr F$ is a full subcategory of $\overline{\mathscr F}$ and since ordinary pullbacks of finite groupoids are again finite groupoids, it therefore only remains to construct the above factorizations.

For this we observe that the model structure on $\cat{Grpd}$ restricts to the structure of a category of fibrant objects in the sense of \cite[I.1]{brown-factorization}, and by Brown's Factorization Lemma (see p.~421 of \emph{op.~cit.}) it then suffices to show that this restricts to also make the category of \emph{finite} groupoids into a category of fibrant objects. However, using again the closure under pullbacks the only non-trivial statement is the existence of path objects, for which we can simply take the standard construction
\begin{equation*}
\mathcal G\xrightarrow{\text{const}}\mathcal G^{[1]}\xrightarrow{(\ev_0,\ev_1)}\mathcal G\times\mathcal G
\end{equation*}
for every finite groupoid $\mathcal G$.
\end{proof}
\end{lemma}

\begin{prop}\label{prop:disjunctive}
The triple $(\mathscr F,\mathscr F_\dagger,\mathscr F)$ is disjunctive in the sense of \cite[Definition~5.2]{barwick-mackey}.
\begin{proof}
We follow the terminology of \emph{loc.~cit.}

The previous lemma shows that $\mathscr F$ has all pullbacks and that they can be computed in terms of ordinary pullbacks. As pullbacks of faithful functors are again faithful by direct inspection, this shows that the above triple is adequate.

Another application of the previous lemma shows that $\mathscr F$ admits all finite coproducts and that they can be computed in $\cat{Grpd}$ again. As a functor out of a coproduct is faithful if and only if it is so on each coproduct summand, and since moreover inclusions of coproduct summands are faithful, this shows that faithful functors are compatible with coproducts.

Finally, let $f\colon I\to K,g\colon J\to K$ be maps of finite sets, and assume we are given for each $(i,j)\in I\times_K J$ a pullback
\begin{equation}\label{diag:pullback-summands}
\begin{tikzcd}
\mathcal A_{i,j}\arrow[r]\arrow[dr,phantom,"\lrcorner"{very near start}]\arrow[d] & \mathcal B_j\arrow[d]\\
\mathcal C_i\arrow[r]&\mathcal D_k
\end{tikzcd}
\end{equation}
in $\mathscr F$ where the horizontal maps are faithful, and where we write $k\coloneqq f(i)=g(j)$. We have to show that the induced square
\begin{equation}\label{diag:pullback-summed}
\begin{tikzcd}
\coprod_{(i,j)\in I\times_KJ}\mathcal A_{i,j}\arrow[r]\arrow[d] & \coprod_{j\in J}\mathcal B_j\arrow[d]\\
\coprod_{i\in I}\mathcal C_i\arrow[r]&\coprod_{k\in K}\mathcal D_k
\end{tikzcd}
\end{equation}
is again a pullback in $\mathscr F$. For this we may assume without loss of generality that each of the diagrams $(\ref{diag:pullback-summands})$ comes from a $1$-categorical pullback of groupoids along an isofibration $\mathcal B_j\to\mathcal D_k$. By direct inspection, the above map $\coprod_{j\in J}\mathcal B_j\to\coprod_{k\in K}\mathcal D_k$ is again an isofibration, so it is enough to show that $(\ref{diag:pullback-summed})$ is a $1$-categorical pullback again. This is immediate by inspecting the standard construction of pullbacks in the $1$-category of groupoids.
\end{proof}
\end{prop}

Specializing \cite[Definition~5.7]{barwick-mackey} we can now introduce our main object of study:

\begin{defi}\label{defi:burnside}
We define $\Agl\coloneqq A^{\textup{eff}}(\mathscr F,\mathscr F_\dagger,\mathscr F)$ and call it the \emph{global effective Burnside category}.
\end{defi}

Explicitly this means that an $n$-simplex of $\Agl$ is a diagram
\[
    \left.\begin{tikzcd}[cramped,column sep=1.1em, row sep=.75em]
        &[-1em]&&& \cdot\arrow[dl]\arrow[dr]\\
        &&&\cdot\arrow[dl]\arrow[dr] && \cdot\arrow[dl]\arrow[dr]\\
        &&\cdot\arrow[dl]\arrow[dr] &&\cdot\arrow[dl]\arrow[dr] &&\cdot\arrow[dl]\arrow[dr]\\
        & \phantom{\cdot} && \phantom{\cdot} && \phantom{\cdot} && \phantom{\cdot}\\[-5ex]
        {\iddots}\!\!\!\! &&&&&&&&[-1.75em]\ddots\\
    \end{tikzcd}\;\right\}\text{$n$ rows of arrows}
\]
in $\mathscr F$, such that all squares are pullbacks and all right-pointing maps are faithful. In particular, a morphism in $A^\text{gl}$ is given by a \emph{span} $\mathcal G\gets\mathcal H\to\mathcal K$ such that the right-pointing map is faithful. We recall from \cite[Notation~5.9]{barwick-mackey} that $A^\text{gl}$ then contains a copy of $\mathscr F^\op$ as the subcategory spanned by those morphisms where the right-pointing morphisms are equivalences, and a copy of $\mathscr F_\dagger$ as the spans whose left-pointing maps are equivalences.

\subsection{Bisets vs.~correspondences of groupoids}
As the first step of our comparison between $(A^\text{gl})^\text{op}$ and $\AglGamma$, we will exhibit the former as an explicit full subcategory $\mathcal A_\text{gl}$ of the $(2,1)$-category of symmetric monoidal categories. Heuristically, the functor $\Phi\colon(A^\text{gl})^\text{op}\to\mathcal A_\text{gl}$ will send a finite group $G$ (viewed as a 1-object groupoid) to the groupoid core of the category of finite free $G$-sets with symmetric monoidal structure coming from disjoint unions; the functoriality in $\mathscr F$ will be given by induction and functoriality in $\mathscr F_\dagger^\op$ via restriction. In this subsection we will construct the restriction of $\Phi$ to $\mathscr F$. Note that while the na\"ive approach would result in a \emph{pseudo\-functor} $\mathscr F\to\cat{SymMonCat}_{(2,1)}$, for technical reasons down the line it will be convenient to have a strict 2-functor instead; this way, we will also avoid having to spell out various 2-categorical coherences by hand. To achieve this, we will replace the category of finite free $G$-sets by a suitable subcategory of the slice (over-category) $\cat{Grpd}\downarrow{BG}$, with the translation given by the 1-categorical Grothendieck construction. We begin by introducing these slice categories and recalling the necessary facts about the Grothendieck construction; in fact, we will do all of this in slightly greater generality, which will later allow us to understand mapping spaces in $A^\text{gl}$ very explicitly as certain groupoids of bisets. Similar comparisons have been given by Miller \cite{miller-burnside} or Dell'Ambrogio and Huglo \cite{della-huglo}, although we unfortunately cannot use their results directly.

\begin{defi}
Let $\mathcal G,\mathcal H$ be groupoids. We call a functor $X\colon\mathcal G\times\mathcal H\to\cat{Set}$ a \emph{$\mathcal G$-$\mathcal H$-biset}. We say that $X$ is \emph{$\mathcal G$-free} if $X(G,H)$ is a free $\Aut_{\mathcal G}(G)$-set for every $G\in\mathcal G,H\in\mathcal H$; equivalently: if $g,g'\colon G\to G'$, $h\colon H\to H'$ are morphisms in $\mathcal G$ and $\mathcal H$, respectively, and $x\in X(G,H)$ with $X(g,h)(x)=X(g',h)(x)$, then $g=g'$.
\end{defi}

\begin{constr}
Let $\mathcal G,\mathcal H$ be groupoids. We write $\overline{\F}_{\mathcal G,\mathcal H}$ for the full $2$-subcategory of the $2$-categorical slice $\cat{Grpd}\downarrow\mathcal G\times\mathcal H$ spanned by those functors $\pi\colon\mathcal X\to\mathcal G\times\mathcal H$ for which the composition $\pr_{\mathcal H}\circ\pi\colon\mathcal X\to\mathcal H$ is faithful.
\end{constr}

\begin{constr}
Let $\mathcal G$ be a groupoid. We recall that the classical \emph{Grothendieck construction} (see e.g.~\cite[§2.1.1]{htt}) provides an equivalence from the $2$-category $\Fun(\mathcal G,\cat{Grpd})$ of (say, strict) functors $\mathcal G\to\cat{Grpd}$, pseudonatural transformations, and modifications to the $2$-categorical slice $\cat{Grpd}\downarrow\mathcal G$. On objects, this is given by sending $F\colon\mathcal G\to\cat{Grpd}$ to the map $\pi\colon\int F\to\mathcal G$, where $\int F$ is the groupoid with objects the pairs $(G\in\mathcal G,X\in F(G))$ and morphisms $(G,X)\to (G',X')$ the pairs $(g\colon G\to G',x\colon F(g)(X)\to X')$; composition is defined in the evident way, and the map $\pi$ is given by projection to the first factor. We will further need the definition of $\int$ on \emph{strictly natural} transformations below: if $\sigma\colon F\Rightarrow G$ is natural, then $\int\sigma\colon\int F\to\int G$ is given on objects by $\int\sigma(G,X)=(G,\sigma_G(X))$ and on morphisms by $\int\sigma(g,X)=\big(g, \sigma_{G'}(x)\colon F(g)(\sigma_G(X))=\sigma_{G'}(F(g)(X))\to\sigma_{G'}(X')\big)$.
\end{constr}

\begin{lemma}\label{lemma:grothendieck-restricted-infinite}
Assume $\mathcal G$ and $\mathcal H$ are groupoids. Then the Grothendieck construction restricts to an equivalence $\Fun^{\textup{$\mathcal G$-free}}(\mathcal G\times\mathcal H,\cat{Set})\to\overline{\F}_{\mathcal G,\mathcal H}$.
\begin{proof}
We first observe that the above is well-defined: if $F\colon\mathcal G\times\mathcal H\to\cat{Set}$ is $\mathcal G$-free, and $(g,h),(g',h')$ are morphisms $(G,H;X)\to (G',H';X')$ in $\int F$ with $h=h'$, then $F(g,h)(X)=X'=F(g',h')(X)=F(g',h)(X)$, so $g=g'$ by freeness; thus $\pr_{\mathcal H}\circ\pi$ is faithful.

On the other hand, if $X$ is a functor such that $\pi\colon\int X\to\mathcal G\times\mathcal H$ is faithful then one easily checks that each $X(G,H)$ is equivalent to a discrete groupoid. If now $\rho\colon\mathcal X\to\mathcal G\times\mathcal H$ is any functor such that $\pr_{\mathcal H}\circ \rho$ is faithful, then there exists a functor $X\colon\mathcal G\times\mathcal H\to\cat{Grpd}$ such that $\rho$ is equivalent to $\pi\colon\int F\to\mathcal G\times\mathcal H$. As also $\pr_{\mathcal H}\circ \pi\colon\int F\to\mathcal H$ is faithful, we can (by the above observation) assume without loss of generality that $X$ factors through $\cat{Set}$. It only remains to show that $X$ is $\mathcal G$-free, which however just follows from running the above argument backwards.
\end{proof}
\end{lemma}

In particular, we see that the $(2,1)$-category $\overline{\F}_{\mathcal G,\mathcal H}$ is actually a $1$-category, i.e.~the quotient map to its homotopy category $\h\overline{\F}_{\mathcal G,\mathcal H}$ is an equivalence.

\begin{rk}
If $X$ is any $\mathcal G$-$\mathcal H$-biset, then $\pi\colon\int X\to\mathcal G\times\mathcal H$ is an isofibration, and if $f\colon X\to Y$ is a natural transformation, then $\int f$ actually strictly commutes with the projections. Thus, the Grothendieck construction factors through the $2$-subcategory $\overline{\F}_{\mathcal G,\mathcal H}^{\textup{iso}}$ whose objects are the isofibrations and whose morphisms are given by \emph{strictly} commuting diagrams; in particular, $\overline{\F}_{\mathcal G,\mathcal H}^{\text{iso}}\hookrightarrow\overline{\F}_{\mathcal G,\mathcal H}$ is an equivalence.
\end{rk}

\begin{rk}\label{rk:straightening-infinite}
Write $\overline{\mathscr D}$ for the full $2$-subcategory of $\cat{Grpd}$ spanned by the essentially discrete groupoids. The usual  \emph{straightening construction} is a quasi-inverse of the Grothen\-dieck construction that turns an object $\pi\colon\mathcal X\to\mathcal G\times\mathcal H$ of $\overline{\F}_{\mathcal G,\mathcal H}$ into a pseudofunctor $\mathcal G\to\overline{\mathscr D}$ and similarly for morphisms. As the $2$-category $\overline{\mathscr D}$ is equivalent to the $1$-category $\cat{Set}$ via the functor $\pi_0$ taking connected components, this then yields a quasi-inverse $\tni\colon\h\overline{\F}_{\mathcal G,\mathcal H}^{\text{iso}}\to\Fun^{\text{$\mathcal G$-free}}(\mathcal G\times\mathcal H,\cat{Set})$ to $\int$ that we can describe explicitly as follows: an isofibration $\rho\colon\mathcal X\to\mathcal G\times\mathcal H$ is sent to the functor that sends $(G,H)\in\mathcal G\times\mathcal H$ to $\pi_0(\rho^{-1}(G,H))$ and a morphism $(g\colon G\to G',h\colon H\to H')$ in $\mathcal G\times\mathcal H$ to the unique map $(\tni\rho)(g,h)\colon\pi_0(\rho^{-1}(G,H))\to \pi_0(\rho^{-1}(G',H'))$ such that there exists for every $X\in\pi_0(\rho^{-1}(G,H))$ and some (hence any) choice of representatives $x$ of $X$ and $y$ of $(\tni \rho)(g)(X)$ a morphism $\chi\colon x\to y$ in $\mathcal X$ with $\rho(\chi)=(g,h)$. Moreover, if $\zeta\colon\mathcal Y\to\mathcal G\times\mathcal H$ is another object and $\alpha\colon\mathcal X\to\mathcal Y$ defines a morphism between them, then $\tni\alpha$ is given by sending $[x]\in\tni\mathcal X(G,H)=\pi_0(\rho^{-1}(G,H))$ to $[\alpha(x)]$.
\end{rk}

We will now restrict our attention to \emph{finite} groupoids $\mathcal G,\mathcal H$. In this case, let us write $\F_{\mathcal G,\mathcal H}\subset\overline{\F}_{\mathcal G,\mathcal H}$ for the full $2$-subcategory spanned by those $\mathcal X\to\mathcal G\times\mathcal H$ for which $\mathcal X$ is \emph{finite}, and similarly define $\F_{\mathcal G,\mathcal H}^{\textup{iso}}\subset\overline{\F}^{\textup{iso}}_{\mathcal G,\mathcal H}$. As any map between finite groupoids factors as a composition of an equivalence and an isofibration \emph{of finite groupoids} (see Lemma~\ref{lemma:limits-vs-homotopy-limits}), the inclusion $\F_{\mathcal G,\mathcal H}^{\textup{iso}}\hookrightarrow{\F}_{\mathcal G,\mathcal H}$ is again an equivalence.

\begin{lemma}\label{lemma:grothendieck-restricted}
The Grothendieck construction $\int$ and the straightening construction $\tni$ define mutually inverse equivalences $\Fun^{\textup{$\mathcal G$-free}}(\mathcal G\times\mathcal H,\cat{FinSet})\rightleftarrows\F_{\mathcal G,\mathcal H}^{\textup{iso}}$.
\begin{proof}
By Lemma~\ref{lemma:grothendieck-restricted-infinite} and Remark~\ref{rk:straightening-infinite} it suffices to observe that both functors restrict accordingly.
\end{proof}
\end{lemma}

\begin{constr}
We write ${\F}_{\mathcal G}\coloneqq{\F}_{\mathcal G,*}$; for simplicity, we agree to take the product $\mathcal G\times *$ to be actually \emph{equal} to $\mathcal G$. With respect to this choice $\F_{\mathcal G}$ is literally \emph{equal} to the $2$-categorical slice ${\mathscr D}\downarrow\mathcal G$, where ${\mathscr D}\subset\mathscr F$ is the full subcategory spanned by the finite essentially discrete groupoids.
\end{constr}

\begin{constr}\label{constr:f-lower-shriek}
We now define a strict $2$-functor $\psi\colon\mathscr F\to\cat{Cat}$ via $\psi(\mathcal G)=\h\F_{\mathcal G}$, with $2$-functoriality given by the usual functoriality of the $2$-categorical slice. For $f\colon\mathcal G\to\mathcal H$ we abbreviate $f_!\coloneqq\psi(f)$.
\end{constr}

\begin{constr}\label{constr:psi-right-adjoints}
Let $f\colon\mathcal G\to\mathcal H$ be a faithful isofibration of finite groupoids. We construct a functor $f^*\colon\h\F_{\mathcal H}\to\h\F_{\mathcal G}$ as follows: an object $\pi\colon\mathcal X\to\mathcal H$ is sent to the left vertical map in the pullback square
\begin{equation}\label{diag:pullback-defining-epsilon}
\begin{tikzcd}
f^*\mathcal X\arrow[d, "f^*\pi"']\arrow[r, "\epsilon"] & \mathcal X\arrow[d,"\pi"]\\
\mathcal G\arrow[r, "f"'] & \mathcal H
\end{tikzcd}
\end{equation}
(note that $\epsilon$ is faithful as a pullback of a faithful functor, so $f^*\mathcal X$ is indeed essentially discrete again). Moreover, if
\begin{equation}\label{diag:morphism-alpha}
\begin{tikzcd}[column sep=small]
\mathcal X\arrow[rr,"\alpha"]\arrow[rd,"\pi"'{name=pi}, bend right=10pt] && \mathcal Y\arrow[ld, "\rho", bend left=10pt]\\
{}\twocell[urr,from=pi, "\scriptstyle\hat\alpha"{yshift=-5pt}] &\mathcal H
\end{tikzcd}
\end{equation}
represents a morphism in $\h\F_{\mathcal H}$, then $f^*[\alpha,\hat\alpha]$ is constructed as follows: we pick for each $(X,G)\in f^*\mathcal X$ an isomorphism $\hat\beta_{X,G}\colon G\to \hat G$ such that $f(\hat\beta_{X,G})=\hat\alpha_X\colon\pi(X)\to\rho\alpha(X)$; note that this is indeed well-defined as $\pi(X)=f(G)$ and since $f$ was assumed to be an isofibration. We now set $\beta(X,G)\coloneqq(\alpha(X),\hat G)$, which is an element of $f^*\mathcal Y$ by definition of $\hat G$. There is then a unique way to extend this to a functor such that the maps $(\id,\hat\beta_{X,G})\colon(\alpha(X),G)\to(\alpha(X),\hat G)$ define a natural isomorphism filling
\begin{equation*}
\begin{tikzcd}
f^*\mathcal X\arrow[r,hook]\arrow[d, "f^*\alpha"'] & \mathcal X\times\mathcal G\arrow[d,"\alpha\times\mathcal G"]\\
f^*\mathcal Y\arrow[r,hook]&\mathcal Y\times\mathcal G,
\end{tikzcd}
\end{equation*}
and we set $f^*[\alpha,\hat\alpha]\coloneqq[\beta,\hat\beta]$; we omit the easy verification that this is independent of choices and makes $f^*$ into a functor. Moreover, one easily checks that the maps $\epsilon$ from $(\ref{diag:pullback-defining-epsilon})$ assemble into a natural transformation $f_!f^*\Rightarrow\id$, and that together with the maps $\eta\colon \mathcal X\to f^*f_!\mathcal X$ induced via the universal property of the pullback from the identity of $\mathcal X$ and the original structure map $\mathcal X\to\mathcal G$, this exhibits $f^*$ as a right adjoint of $f_!$.
\end{constr}

\begin{rk}
If the above diagram $(\ref{diag:morphism-alpha})$ commutes strictly (i.e.~$\hat\alpha$ is the identity), we can pick $\hat\beta$ also to be the identity, so $f^*\alpha$ is just represented by the usual pullback of $\alpha$ along $f$, i.e.~the restriction of $\alpha\times\mathcal G$.
\end{rk}

\begin{lemma}
The functor $\psi$ factors through the $2$-subcategory $\cat{Cat}^\amalg$ of categories with finite coproducts, finite coproduct preserving functors, and all natural transformations. If $f\colon\mathcal G\to\mathcal G'$ is faithful, then $f_!$ has a right adjoint $f^*$, and this adjoint again preserves finite coproducts.
\begin{proof}
It is clear that $\h\F_{\mathcal G}\simeq\cat{FinSet}\downarrow\mathcal G$ has finite coproducts for every $\mathcal G\in\mathscr F$, and that these are created by the forgetful functor, so that $f_!$ preserves finite coproducts for every $f\colon\mathcal G\to\mathcal H$. On the other hand, if $f$ is any faithful functor, then we factor it as an equivalence $i$ followed by a faithful isofibration $p$ (between finite groupoids); then $p_!$ has a right ajdoint by the previous construction, and any quasi-inverse to $i_!$ provides a right adjoint to it. Thus, it only remains to show that the functor $p^*$ from the above construction preserves finite coproducts. For this, we may restrict to the functor $p^*\colon\h\F_{\mathcal H}^{\text{iso}}\to\h\F_{\mathcal G}^{\text{iso}}$. As on these $p^*$ is just given by the ordinary pullback while finite coproducts can be computed in the $1$-category of (essentially discrete) groupoids, this is just the statement that pullbacks in groupoids preserve coproducts, also cf.~the proof of Proposition~\ref{prop:disjunctive} above.
\end{proof}
\end{lemma}

\begin{rk}
We can actually describe the restriction of $f^*$ to $\h\F_{\mathcal G'}^{\text{iso}}\to\h\F_{\mathcal G}^{\text{iso}}$ for arbitrary faithful $f\colon\mathcal G\to\mathcal G'$ as the functor $f^{\text{pb}}$ given by ordinary pullback along $f$: indeed, if $f$ is an isofibration, this was verified in the above proof, and for an equivalence $f$ it is easy enough to check by hand that the composition of $f_!f^{\text{pb}}$ is isomorphic to the inclusion $\h\F_{\mathcal G'}^{\text{iso}}\hookrightarrow\h\F_{\mathcal G'}$, so that precomposing $f^\text{pb}$ with a quasi-inverse to this inclusion yields a quasi-inverse of $f_!\colon\h\F_{\mathcal G}^{\text{iso}}\to\h\F_{\mathcal G'}^{\text{iso}}$ as claimed.
\end{rk}

\begin{lemma}\label{lemma:mathscr-F-coprod}
Let $\mathcal G$ be a finite groupoid, and let $\mathcal G_1,\dots,\mathcal G_n$ be the components of $\mathcal G$. Then $\h\F_\mathcal G$ is equivalent to $\prod_{i=1}^n\h\F_{\mathcal G_i}$ via taking fibers over $\mathcal G_1,\dots,\mathcal G_n$ (i.e.~pullback along the inclusions $\mathcal G_i\hookrightarrow\mathcal G$).
\begin{proof}
We may assume without loss of generality that each $\mathcal G_i$ has only one object $G_i$ and it moreover suffices to prove this for $\h\F_{\mathcal G}^{\textup{iso}}$. A basic computation then shows that the diagram
\begin{equation*}
\begin{tikzcd}
\h\F_{\mathcal G}^{\textup{iso}}\arrow[r]\arrow[d,"\tni"'] &
\prod_{i=1}^n\h\F_{B\Aut(G_i)}^{\textup{iso}}\arrow[d,"\prod_{i=1}^n\tni"]\\
\Fun^{\text{free}}(\mathcal G,\cat{FinSet})\arrow[r, "\cong"'] & \prod_{i=1}^n\Fun^{\text{free}}(B\Aut(G_i),\cat{FinSet})
\end{tikzcd}
\end{equation*}
commutes strictly, where the top horizontal arrow is as above and the lower one is induced by the inclusions. The claim now follows from Lemma~\ref{lemma:grothendieck-restricted}.
\end{proof}
\end{lemma}

\begin{constr}\label{constr:symmetric-monoidal-core}
We equip $\h\F_{\mathcal G}$ with `the' cocartesian symmetric monoidal structure; once we have fixed such a choice of coproducts, $\psi$ factors uniquely through the $(2,2)$-category $\cat{SymMonCat}_{(2,2)}$ of small symmetric monoidal categories, symmetric monoidal functors, and symmetric monoidal transformations; similarly the above adjoints can be uniquely made into symmetric monoidal functors such that they are adjoints in $\cat{SymMonCat}_{(2,2)}$. In particular, this then induces a symmetric monoidal structure on $\core\h\F_{\mathcal G}$, $f_!$, and $f^*$ (assuming the latter exists).

We moreover define a functor $\tau=\tau_{\mathcal G}\colon\mathcal G\to\core\h\F_{\mathcal G}$ as follows: an object $G\in\mathcal G$ is sent to the map $G\colon *\to\mathcal G$ classifying $G$, and a map $g\colon G\to G'$ is sent to
\begin{equation*}
\begin{tikzcd}[column sep=1.5em]
*\arrow[dr, bend right=15pt, "G"'{name=G}]\arrow[rr,"="] &&*\arrow[dl, bend left=15pt, "G'"] \\
\twocell[urr, from={G}, "\scriptstyle g"{yshift=-5pt}] & \mathcal G.
\end{tikzcd}
\end{equation*}
\end{constr}

\begin{lemma}\label{lemma:F-G-corep}
Let $\mathcal C$ be any symmetric monoidal category. Then restricting along $\tau$ defines an equivalence $\Fun^{\otimes}(\core\h\F_{\mathcal G},\mathcal C)\to\Fun(\mathcal G,\mathcal C)$, where the left hand side denotes the category of strong symmetric monoidal functors and symmetric monoidal transformations.
\end{lemma}

For the proof of the lemma, we first recall the permutative groupoid $\mathfrak F_G$ from Construction~\ref{constr:mathfrak-FG}. The following observation is well-known, but we provide an argument for completeness:

\begin{lemma}\label{lemma:F-frak}
Let $G$ be a group and let $\mathcal C$ be any symmetric monoidal category. Then evaluating at $\cat{1}\in\mathfrak F_G$ defines an equivalence $\Fun^\otimes(\mathfrak F_G,\mathcal C)\to G\text{--}\mathcal C$, where the right hand side denotes the category of $G$-objects.
\begin{proof}
    A direct computation shows that $\mathfrak F_G$ represents the functor $\Cc\mapsto\Ob(G\text{--}\mathcal C)$ on the 1-category $\cat{PermCat}$ of permutative categories and \emph{strict} symmetric monoidal functors, via evaluation at \textbf{1}. We will now explain how deduce the analogous statement for the homotopy category of $\cat{SymMonCat}_{(2,1)}$; the claim will then follow by replacing $\Cc$ by $\Cc^{\mathcal K}$ for any plain category $\mathcal K$ and appealing to the Yoneda Lemma.

    By \cite[Theorem~3.1]{sharma}, $\cat{PermCat}$ admits a simplicial model structure where a map is a weak equivalence or fibration if and only if it is so in the usual model structure on $\cat{Cat}$. Our computation of the corepresented functor then shows that any lifting problem
    \[
        \begin{tikzcd}
            & \Cc\arrow[d, "p", two heads, "\sim"']\\
            \mathfrak F_G\arrow[ur, bend left=15pt,dashed]\arrow[r] & \Dd
        \end{tikzcd}
    \]
    where $p$ is an acyclic fibration (i.e.~an equivalence that is strictly surjective on objects) has a solution, i.e.\ $\mathfrak F_G$ is cofibrant with respect to this model structure. As in addition every object is fibrant, we conclude that the map
    \[
        \Hom_{\cat{PermCat}}(\mathfrak F_G,\Cc)\to
        \Hom_{\Ho(\cat{PermCat})}(\mathfrak F_G,\Cc)
    \]
    induced by the localization functor is just given by dividing out the model categorical notion of homotopy. As the model structure is simplicial, a path object is given by $\Cc^{J}$ where $J$ is the contractible groupoid with two objects; by another application of our copresentability statement, we see that two functors $\mathfrak F_G\to\Cc$ are homotopic if and only if they correspond to isomorphic $G$-objects.

    As the inclusion $\cat{PermCat}\hookrightarrow\cat{SymMonCat}$ induces an equivalence on homotopy categories by Mac Lane's strictification theorem (cf.\ \cite[Theorem~1.19]{perm-parsum-categorical}), we altogether conclude that evaluation at $\textbf{1}$ induces a bijection
    \[
        \Hom_{\Ho(\cat{SymMonCat})}(\mathfrak F_G,\Cc)\cong\{
            \text{isomorphism classes of $G$-objects in $\Cc$}\}
    \]
    for any permutative $\Cc$, and hence (using Mac Lane's strictification theorem again) for any symmetric monoidal $\Cc$. As explained above, this implies the lemma.
\end{proof}
\end{lemma}

\begin{proof}[Proof of Lemma~\ref{lemma:F-G-corep}]
As $\tau$ is natural with respect to pushforward, we may assume without loss of generality that each component of $\mathcal G$ consists of a single object, and by Lemma~\ref{lemma:mathscr-F-coprod} we can then reduce to the case that $\mathcal G=BG$ for some finite group $G$; more precisely, if $G_1,\dots,G_n$ are the (pairwise non-isomorphic) objects of $\mathcal G$, then the diagram
\begin{equation*}
\begin{tikzcd}
\coprod_{i=1}^n BG_i\arrow[d,"\cong"']\arrow[r,"\diag(\tau)"] &[.67em]\prod_{i=1}^n\h\F_{BG_i}\\
\mathcal G\arrow[r,"\tau"']&\h\F_{\mathcal G}\arrow[u,"\simeq"']
\end{tikzcd}
\end{equation*}
with the equivalence from Lemma~\ref{lemma:mathscr-F-coprod} on the right commutes strictly for the usual construction of fibers, and the $(2,1)$-category $\cat{SymMonCat}$ is semiadditive.

But for $\mathcal G=BG$, we have an obvious equivalence $\mathfrak F_G\to\core\h\F_{BG}$ compatible with the maps from $BG$, so the claim follows immediately from the previous lemma.
\end{proof}

\subsection{From \texorpdfstring{$\bm A^{\textup{gl}}$}{Agl} to symmetric monoidal categories} Our next goal is to extend $\nerve_\Delta(\psi)\colon\mathscr F\to\nerve_\Delta(\cat{Cat}^\amalg)$ to a functor on $(\Agl)^\op$ using the adjoints we constructed above. This is an instance of a general $2$-categorical construction:

\begin{prop}\label{prop:unfurl}
Let $\mathscr I$ be a strict $(2,1)$-category, and let $\mathscr I^\dagger\subset\mathscr I$ be a $2$-subcategory such that $(\nerve_\Delta(\mathscr I),\nerve_\Delta(\mathscr I),\nerve_\Delta(\mathscr I^\dagger))$ is an adequate triple. Let $\mathscr C$ be any strict $2$-category and write $\mathscr C_{(2,1)}$ for its underlying $(2,1)$-category, i.e.~the $2$-category obtained by throwing away all non-invertible $2$-cells. Moreover, let $\phi\colon\mathscr I\to\mathscr C$ be a strict $2$-functor such that for every $i\in\mathscr I^\dagger$ the functor $i_!\coloneqq\phi(i)$ admits a right adjoint $i^*$ and such that for every homotopy pullback diagram in $\nerve_\Delta(\mathscr I)$ as on the left in
\begin{equation*}
\begin{tikzcd}
A\arrow[r,"g"]\arrow[d,"j"',tail]&\twocell[dl, "\scriptstyle\sigma"{yshift=-8pt},yshift=3.5pt] B\arrow[d,"i",tail]\\
C\arrow[r,"f"'] & D
\end{tikzcd}\hskip1in
\begin{tikzcd}
\phi(A)\arrow[r,"g_!"]\twocell[dr, "\scriptstyle\sigma_\lozenge"{yshift=-8pt},yshift=3.5pt] &\phi(B)\\
\phi(C)\arrow[u,"j^*"]\arrow[r,"f_!"'] & \phi(D)\arrow[u,"i^*"']
\end{tikzcd}
\end{equation*}
(which we drew as a diagram in $\mathscr I$ by omitting the diagonal edge $\phi(A)\to\phi(D)$ and already pasting the two natural isomorphisms) with vertical arrows belonging to $\mathscr I^\dagger$, the canonical mate $\sigma_\lozenge$ of $\sigma_!$ depicted on the right is an isomorphism.

Then there is a unique functor $\Phi\colon A^{\textup{eff}}(\nerve_\Delta(\mathscr I),\nerve_\Delta(\mathscr I),\nerve_\Delta(\mathscr I^\dagger))\to\nerve_\Delta(\mathscr C_{(2,1)})$ that sends a $2$-simplex of the form
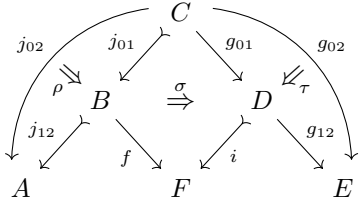
\begin{equation*}
\begin{tikzcd}[column sep=1.35em,row sep=1.875em]
&& C\arrow[ddll, bend right=40pt, shift right=5pt, "j_{02}"',""{name=j02}]\arrow[ddrr, bend left=40pt, shift left=5pt, "g_{02}", ""'{name=g02}]\arrow[dl, "j_{01}"',tail]\arrow[dr, "g_{01}"]\\
& \twocell[from=j02, "\scriptstyle\rho"{yshift=-5pt,xshift=-5pt}]B\arrow[dl, "j_{12}"',tail]\arrow[dr, "f"']\twocell[rr,"\scriptstyle\sigma"{yshift=6pt}] && D\arrow[dl, "i",tail]\arrow[dr, "g_{12}"]\twocell[from=g02,"\scriptstyle\tau"{yshift=-5pt,xshift=5pt}]\\
A && F && E
\end{tikzcd}
\end{equation*}
(where we have again omitted the edge $C\to F$ and already pasted the two transformations filling the middle square) to the $2$-simplex
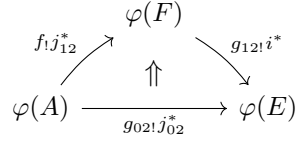
\begin{equation*}
\begin{tikzcd}[column sep=small]
& \phi(F)\arrow[dr, bend left=10pt, "g_{12!}i^*"]\\
\phi(A)\arrow[rr,"g_{02!}j_{02}^*"'{name=bot}]\arrow[ur, bend left=10pt, "f_!j_{12}^*"]\twocell[from=bot,ur] && \phi(E)
\end{tikzcd}
\end{equation*}
given as the pasting
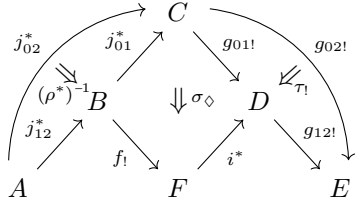
\begin{equation*}
\begin{tikzcd}[column sep=1.35em,row sep=1.875em]
&& \twocell[dd,"\scriptstyle\sigma_\lozenge"{xshift=10pt}]C\arrow[from=ddll, bend left=40pt, shift left=5pt, "j_{02}^*",""{name=j02}]\arrow[ddrr, bend left=40pt, shift left=5pt, "g_{02!}", ""'{name=g02}]\arrow[from=dl, "j_{01}^*"]\arrow[dr, "g_{01!}"]\\
& \twocell[from=j02, "\scriptstyle(\mskip-.3\thinmuskip\rho^*\mskip-.3\thinmuskip)^{\!{-}\!1}"{yshift=-7pt,xshift=-1pt},xshift=2pt,yshift=-2pt]B\arrow[from=dl, "j_{12}^*"]\arrow[dr, "f_!"'] && D\arrow[from=dl, "i^*"']\arrow[dr, "g_{12!}"]\twocell[from=g02,"\scriptstyle\tau_!"{yshift=-5pt,xshift=5pt}]\\
A && F && E
\end{tikzcd}
\end{equation*}
where $\rho^*$ is the total mate of $\rho_!$ and $\sigma_\lozenge$ is defined as above.
\end{prop}

The proof of the proposition is deferred to the appendix. As we will only apply this to the functor $\psi$ landing in $\cat{Cat}^\amalg$, and since we will ultimately be only interested in the postcomposition of this with $\core\colon\cat{Cat}^\amalg_{(2,1)}\to\cat{SymMonGrpd}$, we could have also instead applied Barwick's \emph{unfurling construction} for Waldhausen bicartesian fibrations \cite[Section~11]{barwick-mackey}. However, Barwick's result first yields a (Waldhausen) cocartesian fibration, which one then has to straighten into a functor, and as we will at several places below need the above explicit description of the resulting functor on $2$-simplices, proving the $2$-categorical proposition directly seems to be less work than unravelling Barwick's construction.

We moreover remark that the analogous statement for the `classical' span \hbox{$2$-category} (and without any of the above strictness assumptions) appears as \cite[Theorem~5.2.1]{balmer-mackey}; thus, if one is willing to believe that $\Agl$ is equivalent to the Duskin nerve of the corresponding span $2$-category in a compatible way, one could as well deduce this from \emph{loc.~cit.}

In order to apply the above proposition to our context we note:

\begin{lemma}
Let
\begin{equation}\label{diag:F-homotopy-pull-base}
\begin{tikzcd}
\mathcal A\arrow[r,"g"]\arrow[d,"j"',tail] &\twocell[dl,"\scriptstyle\sigma"{yshift=-7pt},yshift=2pt] \mathcal B\arrow[d, "i",tail]\\
\mathcal C\arrow[r,"f"']&\mathcal D
\end{tikzcd}
\end{equation}
be a homotopy pullback in $\mathscr F$ such that the vertical functors are faithful. Then the mate transformation $\sigma_\lozenge\colon g_!j^*\Rightarrow i^*f_!$ is an isomorphism of functors $\h\F_{\mathcal C}\to\h\F_{\mathcal B}$.
\begin{proof}
The claim is clear if $i$ is an equivalence; by the compatibility of mates with pastings we may therefore assume without loss of generality that $\sigma$ is the identity and that $(\ref{diag:F-homotopy-pull-base})$ is an ordinary pullback along the isofibration $i$. In this case, we immediately see using the explicit description of the adjunctions given in Construction~\ref{constr:psi-right-adjoints} that for an object $\pi\colon\mathcal X\to\mathcal B$, $(\id_\lozenge)_{\mathcal X}$ is represented by the canonical comparison map $\mathcal X\times_{\mathcal B}\mathcal A=\mathcal X\times_{\mathcal B}(\mathcal B\times_{\mathcal D}\mathcal C)\to \mathcal X\times_{\mathcal D}\mathcal C$, which is even an isomorphism of groupoids.
\end{proof}
\end{lemma}

Thus, we may apply Proposition~\ref{prop:unfurl} to extend $\psi$ to a functor $\Psi\colon(\Agl)^\op\cong A^{\text{eff}}(\mathscr F,\mathscr F,\mathscr F_\dagger)\to\nerve_\Delta(\cat{Cat}^\amalg_{(2,1)})$. As announced long ago, we want to prove:

\begin{thm}\label{thm:Agl-ord-to-mathcal}
The composition $\core\circ\Psi\colon (\Agl)^\op\to\nerve_\Delta(\cat{SymMonCat}_{(2,1)})$ is fully faithful.
\end{thm}

We denote the essential image of the above functor by $\mathcal A_{\textup{gl}}$ (and by slight abuse of notation, we will also denote the corresponding $2$-subcategory of $\cat{SymMonCat}_{(2,1)}$ by the same symbol); by Lemma~\ref{lemma:mathscr-F-coprod} it consists precisely of those symmetric monoidal categories that are equivalent to finite products of copies of $\core\h\F_{BH}$ or equivalently $\mathfrak F_H$ for varying finite groups $H$.

The proof of the theorem requires some preparations.

\begin{constr}
Let $\mathcal G,\mathcal H$ be finite groupoids. We define a functor
\begin{equation*}
\Psi'\colon\core\h\F_{\mathcal G,\mathcal H}\to\Fun^{\amalg}(\h\F_{\mathcal H},\h\F_{\mathcal G})
\end{equation*}
as follows: an object $\pi\colon\mathcal X\to\mathcal G\times\mathcal H$ is sent to the functor $\pi_{\mathcal G!}\pi_{\mathcal H}^*$ (where $\pi_{\mathcal G},\pi_{\mathcal H}$ denote the components of $\pi$), and the class of a morphism
\begin{equation*}
\begin{tikzcd}[column sep=small]
\mathcal X\arrow[rr,"\alpha"]\arrow[rd,"\pi"'{name=pi}, bend right=10pt] && \mathcal Y\arrow[ld, "\rho", bend left=10pt]\\
{}\twocell[urr,from=pi, "\scriptstyle\hat\alpha"{yshift=-5pt}] &\mathcal H
\end{tikzcd}
\end{equation*}
is sent to the pasting
\begin{equation*}
\begin{tikzcd}[column sep=2.2em, row sep=2.2em]
\h\F_{\mathcal H}\arrow[d,"="']\twocell[dr, "\scriptstyle(\hat\alpha_{\mathcal H}^{-1})_\lozenge"{yshift=-10pt},yshift=4pt]\arrow[r,"\pi_{\mathcal H}^*"] & \h\F_{\mathcal X}\twocell[dr,"\scriptstyle\hat\alpha_{\mathcal G!}"{yshift=-10pt},yshift=4pt]\arrow[r,"\pi_{\mathcal G!}"]\arrow[d,"\alpha_!"] & \h\F_{\mathcal G}\arrow[d,"="]\\
\h\F_{\mathcal H}\arrow[r,"\rho_{\mathcal H}^*"'] & \h\F_{\mathcal Y}\arrow[r, "\rho_{\mathcal G!}"'] & \h\F_{\mathcal G}
\end{tikzcd}
\end{equation*}
where $(\hat\alpha_{\mathcal H}^{-1})_\lozenge$ again denotes the canonical mate. We omit the easy verification that this is well-defined and a functor.
\end{constr}

\begin{prop}
The composition $\core\circ\Psi'$ defines an equivalence
\begin{equation*}
\core\h\F_{\mathcal G,\mathcal H}\simeq\Fun^\otimes(\core\h\F_{\mathcal H},\core\h\F_{\mathcal G}).
\end{equation*}
\begin{proof}
It suffices to show this after restricting to $\core\h\F_{\mathcal G,\mathcal H}^{\textup{iso}}$. In this case, the right adjoints $\pi_{\mathcal H}^*$ appearing in the construction of $\Psi'$ can just be taken to be the ones from Construction~\ref{constr:psi-right-adjoints}, i.e.~they are given on objects by ordinary pullback. We now agree on a specific choice of these pullbacks in one instance: namely, for any object of $\h\F_{\mathcal H}$ of the form $\zeta\colon *\to\mathcal H$ and any object $\pi\colon\mathcal X\to\mathcal G\times\mathcal H$ of $\h\F_{\mathcal G,\mathcal H}^\text{iso}$ we agree to take the pullback defining $\pi_{\mathcal H}^*(\zeta)$ as the ordinary fiber of $\pi_{\mathcal H}$ over $\zeta(*)$, i.e.~the subgroupoid $\pi_{\mathcal H}^{-1}(\zeta(*))\subset\mathcal X$, with structure map given by the inclusion. With these conventions, a basic computation then shows that the diagram
\begin{equation*}
\begin{tikzcd}
\core\h\F_{\mathcal G,\mathcal H}\arrow[d,"\tni"']\arrow[r,"\core\circ\Psi'"] & \Fun^\otimes(\core\h\F_{\mathcal H},\core\h\F_{\mathcal G})\arrow[dd,"\tau^*"]\\
\core\Fun^{\text{$\mathcal G$-free}}(\mathcal G\times\mathcal H,\cat{FinSet})\arrow[d,"\cong"']\\
\core\Fun(\mathcal H,\Fun^{\text{free}}(\mathcal G,\cat{FinSet})) & \arrow[l,"\tni"] \Fun(\mathcal H,\core\h\F_{\mathcal G})
\end{tikzcd}
\end{equation*}
actually commutes strictly, where the unlabelled isomorphism is given by currying. The claim now follows from Lemma~\ref{lemma:grothendieck-restricted} together with Lemma~\ref{lemma:F-G-corep}.
\end{proof}
\end{prop}

\begin{proof}[Proof of Theorem~\ref{thm:Agl-ord-to-mathcal}]
Let $\mathcal G,\mathcal H\in\mathscr F$. We have to show that $\Psi$ induces a homotopy equivalence
\begin{equation}\label{eq:hom-spaces-Psi}
\Hom^{\textup{R}}_{A^{\text{eff}}(\mathscr F,\mathscr F,\mathscr F_\dagger)}(\mathcal H,\mathcal G)\to\Hom^{\textup{R}}_{\nerve_\Delta(\cat{SymMonCat}_{(2,1)})}(\core\h\F_{\mathcal H},\core\h\F_{\mathcal G})
\end{equation}
of the right mapping spaces from \cite[p.\ 27]{htt}. The right hand side is clearly a $1$-truncated Kan-complex; on the other hand, the left hand side is a $2$-truncated Kan-complex (as a subcategory of a category of functors into a $2$-category), and we claim that it is in fact again $1$-truncated; once we know this, it will be enough to show that $(\ref{eq:hom-spaces-Psi})$ induces an equivalence of homotopy categories.

For this, let us first describe the mapping space on the left explicitly: its objects are the spans $\mathcal H\xleftarrow{i} \mathcal X\xrightarrow{f} \mathcal G$ such that $i$ is faithful, and a morphism from such a span to another object $\mathcal H\xleftarrow{j} \mathcal Y\xrightarrow{g} \mathcal G$ is given by a diagram as on the left in
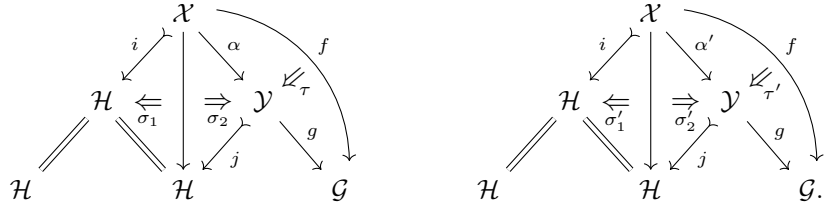
\begin{equation*}
\begin{tikzcd}[column sep=1.35em,row sep=1.875em]
&& \mathcal X\arrow[dd]\arrow[ddrr, bend left=40pt, shift left=5pt, "f", ""'{name=g02}]\arrow[dl, "i"',tail]\arrow[dr, "\alpha"]\\
& \mathcal H\arrow[dl, equal]\twocell[from=r,"\scriptstyle\sigma_1"{yshift=-7pt}]\arrow[dr, equal] &{}& \twocell[from=l,"\scriptstyle\sigma_2"{yshift=-7pt}]\mathcal Y\arrow[dl, "j",tail]\arrow[dr, "g"]\twocell[from=g02,"\scriptstyle\tau"{yshift=-5pt,xshift=5pt}]\\
\mathcal H&& \mathcal H && \mathcal G
\end{tikzcd}
\qquad\qquad
\begin{tikzcd}[column sep=1.35em,row sep=1.875em]
&& \mathcal X\arrow[dd]\arrow[ddrr, bend left=40pt, shift left=5pt, "f", ""'{name=g02}]\arrow[dl, "i"',tail]\arrow[dr, "\alpha'"]\\
& \mathcal H\arrow[dl, equal]\twocell[from=r,"\scriptstyle\sigma_1'"{yshift=-7pt}]\arrow[dr, equal] &{}& \twocell[from=l,"\scriptstyle\sigma_2'"{yshift=-7pt}]\mathcal Y\arrow[dl, "j",tail]\arrow[dr, "g"]\twocell[from=g02,"\scriptstyle\tau'"{yshift=-5pt,xshift=5pt}]\\
\mathcal H&& \mathcal H && \mathcal G\rlap{.}
\end{tikzcd}
\end{equation*}
By direct inspection using faithfulness of the contravariant legs, for any other such diagram (as depicted on the right), there is at most one homotopy between them, and such a homotopy exists if and only if there is an isomorphism $\phi\colon\alpha\cong\alpha'$ such that the pasting of $\phi$ with $\rho$ agrees with $\rho'$ and the pasting of $\phi$ with $(\sigma'_2)(\sigma'_1)^{-1}$ agrees with $\sigma_2\sigma_1^{-1}$. Thus, $\Hom^{\textup{R}}_{A^{\text{eff}}(\mathscr F,\mathscr F,\mathscr F_\dagger)}$ is $1$-truncated and we have an isomorphism of categories
\begin{equation*}
\core\h\F_{\mathcal G,\mathcal H}\to \h\Hom^{\textup{R}}_{A^{\text{eff}}(\mathscr F,\mathscr F,\mathscr F_\dagger)}(\mathcal H,\mathcal G)
\end{equation*}
sending $\pi\colon\mathcal X\to\mathcal G\times\mathcal H$ to the span $\mathcal H\xleftarrow{\pi_{\mathcal H}}\mathcal X\xrightarrow{\pi_{\mathcal G}}\mathcal G$, and a morphism from $\pi$ to $\rho\colon\mathcal Y\to\mathcal G\times\mathcal H$ represented by an equivalence $\alpha\colon\mathcal X\to\mathcal Y$ together with an isomorphism $\hat\alpha\colon \pi\cong\rho\alpha$ to the class of
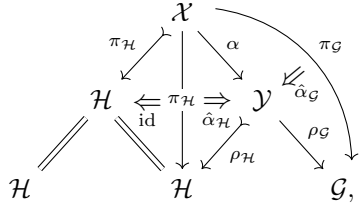
\begin{equation*}
\begin{tikzcd}[column sep=1.35em,row sep=1.875em]
&& \mathcal X\arrow[dd, "\pi_{\mathcal H}"{description}]\arrow[ddrr, bend left=40pt, shift left=5pt, "\pi_{\mathcal G}", ""'{name=g02}]\arrow[dl, "\pi_{\mathcal H}"',tail]\arrow[dr, "\alpha"]\\
& \mathcal H\arrow[dl, equal]\arrow[dr, equal]&\twocell[l,"\scriptstyle\id"{yshift=-7pt}]{}\twocell[r,"\scriptstyle\hat\alpha_{\mathcal H}"{yshift=-7pt}] & \mathcal Y\arrow[dl, "\rho_{\mathcal H}",tail]\arrow[dr, "\rho_{\mathcal G}"]\twocell[from=g02,"\scriptstyle\hat\alpha_{\mathcal G}"{yshift=-6pt,xshift=6pt}]\\
\mathcal H&& \mathcal H && \mathcal G,
\end{tikzcd}
\end{equation*}
also cf.~\cite[3.7]{barwick-mackey}, which states (without proof) the existence of an equivalence between the mapping spaces of the effective Burnside category of a general $\infty$-category $\mathscr I$ with all pullbacks and the cores of certain slices of $\mathscr I$.

Similarly (but much easier), we have an isomorphism
\begin{equation*}
\h\Hom^{\text{R}}(\core\h\F_{\mathcal H},\core\h\F_{\mathcal G})\cong\Fun^\otimes(\core\h\F_{\mathcal H},\core\h\F_{\mathcal G})
\end{equation*}
that is the identity on objects and sends the class of
\begin{equation*}
\begin{tikzcd}[column sep=small]
& \core\h\F_{\mathcal H}\arrow[dr, bend left=10pt, "g"]\\
\core\h\F_{\mathcal H}\arrow[rr,"f"'{name=bot}]\arrow[ur, bend left=10pt, equal]\twocell[from=bot,ur,"\scriptstyle\sigma"{xshift=8pt}] && \core\h\F_{\mathcal G}
\end{tikzcd}
\end{equation*}
simply to $\sigma\colon f\Rightarrow g$. Using the explicit description of $\Psi$ on $2$-simplices from Proposition~\ref{prop:unfurl}, we then see that the resulting composition $\core\h\F_{\mathcal G,\mathcal H}\cong
\h\Hom^{\textup{R}}(\mathcal H,\mathcal G)\to\h\Hom^{\textup{R}}(\core\h\F_{\mathcal H},\core\h\F_{\mathcal G})\cong\Fun^\otimes(\core\h\F_{\mathcal H},\core\h\F_{\mathcal G})$ is exactly $\core\circ\Psi'$, so it is an equivalence by the previous proposition, and hence so is $(\ref{eq:hom-spaces-Psi})$ as desired.
\end{proof}

\subsection{Special global \texorpdfstring{$\bm\Gamma$}{Γ}-spaces vs.\ global Mackey functors} We have now done most of the hard work necessary to prove our first main result:

\begin{thm}\label{thm:ucom}
There is an equivalence of $\infty$-categories
\begin{equation*}
(\cat{$\bm\Gamma$-$\bm{E\mathcal M}$-SSet}_*^{\textup{special}})^\infty\simeq\Fun^\times(\Agl,\mathscr S).
\end{equation*}
\end{thm}

The only missing ingredient is the following:

\begin{prop}
    For each finite group $G$ and every symmetric monoidal \emph{groupoid} $\Cc$ the functor $\Gamma_\textup{gl}\colon\nerve_\Delta(\cat{SymMonCat}_{(2,1)})\to(\cat{$\bm\Gamma$-$\bm{E\mathcal M}$-SSet}_*^\textup{special})^\infty$ from Construction~\ref{constr:segal-may-shimada-shimakawa} induces an equivalence of mapping spaces
    \[
        \maps(\mathfrak F_G,\Cc)\simeq\maps(\Gamma_\textup{gl}\mathfrak F_G,\Gamma_\textup{gl}(\Cc)).
    \]
    In particular, $\Gamma_\textup{gl}$ is fully faithful when restricted to $\mathcal A_\textup{gl}$.
    \begin{proof}
        We may assume without loss of generality that $G$ is a universal subgroup of $\mathcal M$. Then Theorem~\ref{thm:fixed-points-corep} shows that $\Gamma_\textup{gl}(\mathfrak F_G)$ corepresents the functor sending a special global $\Gamma$-space $X$ to $X(1^+)^H$ via evaluation at a certain explicit element $\tau$.

        On the other hand, Lemma~\ref{lemma:F-G-corep} shows that $\mathfrak F_G$ corepresents the functor $\Cc\mapsto\nerve(G\text{--}\Cc)$ on $\nerve_\Delta(\cat{SymMonCat}_{(2,1)})$ via evaluation at the $G$-object $\rho\colon BG\to\mathfrak F_G$ sending $g$ to $g\colon\bm1\to\bm1$.

        Finally, we have an equivalence $\Fun(E\mathcal M,\Cc)^G\to G\text{-}\Cc$ sending a $G$-fixed functor $F$ to $F(1)$ with $g\in G$ acting via $F(g,1)$, and sending a $G$-fixed transformation $\tau\colon F\Rightarrow F'$ to $\tau_1\colon F(1)\to F'(1)$.

        A straight-forward computation then shows that the composite
        \[
            \maps(\mathfrak F_G,\Cc)\simeq \nerve(G\text{--}\Cc)\simeq\nerve\Fun(E\mathcal M,\Cc)^G\simeq\maps(\Gamma_\text{gl}\mathfrak F_G,\Gamma_\text{gl}\Cc)
        \]
        of natural transformations sends the identity of $\mathfrak F_G$ to the identity of $\Gamma_\text{gl}\mathfrak F_G$. The same holds for the map induced by $\Gamma_\text{gl}$, so the latter has to agree with the above composite by the Yoneda Lemma; in particular, it is an equivalence, as claimed.
    \end{proof}
\end{prop}

\begin{proof}[Proof of Theorem~\ref{thm:ucom}]
By Proposition~\ref{prop:gamma-mackey-geometric} we have an equivalence
\[
    (\cat{$\bm\Gamma$-$\bm{E\mathcal M}$-SSet}_*^\text{special})^\infty\simeq\Fun^\times(\AglGamma^\op,\mathscr S),
\]
while Theorem~\ref{thm:Agl-ord-to-mathcal} together with the previous proposition provides a chain of equivalences $A^\text{gl}\simeq\mathcal A_\text{gl}^\op\simeq\AglGamma^\op$; the claim follows immediately.
\end{proof}

\section{Global spectral Mackey functors vs.~global spectra}\label{sec:spectra}
Building on the above results, we will prove Theorem~\ref{introthm:mackey-vs-spectra} on the comparison between global spectra and spectral Mackey functors on $\Agl$ in this section.

\subsection{A reminder on global spectra}
Our reference model of global stable homotopy theory is Hausmann's global model structure on the category of \emph{symmetric spectra} in the sense of \cite{hss}. In order to define this, we need:

\begin{defi}
A symmetric spectrum $X$ is called a \emph{global $\Omega$-spectrum} if for every finite group $H$, every finite faithful $H$-set $A$, and every finite $H$-set $B$ the derived adjoint structure map
\begin{equation*}
X(A)^H\to\big(\cat{R}\Omega^BX(A\amalg B)\big)^H
\end{equation*}
is a weak homotopy equivalence. Here we are deriving $\Omega$ with respect to the usual equivariant model structure on $\cat{$\bm H$-SSet}$, for example by precomposing with the singular set-geometric realization adjunction.
\end{defi}

\begin{thm}
There is a unique model structure on the category $\cat{Spectra}$ of symmetric spectra in which a map $f$ is a weak equivalence or fibration if and only if $f(A)^H$ is a weak homotopy equivalence or Kan fibration, respectively, for all finite groups $H$ and all finite faithful $H$-sets $A$. We call this the \emph{global level model structure}.

Furthermore, the global level model structure admits a Bousfield localization whose fibrant objects are precisely the level fibrant global $\Omega$-spectra. We call the resulting model structure the \emph{global model structure} and its weak equivalences the \emph{global weak equivalences}.
\begin{proof}
See \cite[Proposition~2.5 and Theorem~2.17]{hausmann-global}.
\end{proof}
\end{thm}

By \cite[Proposition~4.6-(i)]{hausmann-global} the above model structure is stable, and hence so is the $\infty$-category $\Sp_\textup{gl}\coloneqq\cat{Spectra}_\textup{global}^\infty$.

\begin{constr}
    Let $H$ be a finite group and let $\mathcal U_H$ be {complete $H$-set universe} (see Definition~\ref{def:universe}). For any $k\ge0$ and any symmetric spectrum $X$ we define the \emph{na\"ive} homotopy group
    \[
        \tilde\pi_k^H(X)\coloneqq\colim\limits_{A\subset\mathcal U_H\atop\textup{finite $H$-set}} [|S^A\smashp S^k|,|X(A)|]_*
    \]
    where $|{\cdot}|$ denotes the geometric realization, $[\,{,}\,]_*$ denotes based homotopy classes, and the transition maps are induced by the structure maps of $X$ in the evident way. Similarly, we define
    \[
        \tilde\pi_k^H(X)\coloneqq\colim\limits_{A\subset\mathcal U_H\atop\textup{finite $H$-set}} [|S^A|,|X(\{1,2,\dots,-k\}\amalg A)|]_*
    \]
    for $k<0$.
\end{constr}

The central downside of the symmetric spectrum model is that the above na\"ively defined homotopy groups are not invariant under global weak equivalences. We therefore define the \emph{true homotopy groups} $\pi_*^H(X)$ as the na\"ive homotopy groups of a fibrant replacement in the global model structure. As a consequence of \cite[Lemma~3.1.33]{g-global} the true homotopy groups are indepdent of the choice of fibrant replacement and invariant under global weak equivalences.

\begin{defi}
    A global spectrum $X$ is called \emph{connective} if $\pi_k^H(X)=0$ for every $k<0$ and every finite group $H$.
\end{defi}

\begin{prop}\label{prop:t-structure}
    There exists a (unique) right-complete t-structure on $\Sp_\textup{gl}$ with connective part $\Sp^{\ge0}_\textup{gl}$ given by the connective global spectra.
    \begin{proof}
        See \cite[Theorem~7.1.12]{CLL_Global}.\footnote{Note that despite appearing in \cite{CLL_Global}, this has of course nothing to do with parametrized higher category theory, and instead follows from some general nonsense about compactly generated stable $\infty$-categories.}
    \end{proof}
\end{prop}

\subsection{The comparison} We now turn to the proof of our comparison result. The key idea is the following simple observation:

\begin{lemma}\label{lemma:Mackey-stab}
    The functor $\cat{R}\Omega^\infty\colon\Fun^\times(\Agl,\Sp)\to\Fun^\times(\Agl,\mathscr S)$ exhibits the source as the stabilization of the target, i.e.~for every stable $\infty$-category $\Cc$ the induced map
    \[
        \Fun^\textup{lex}\big(\Cc, \Fun^\times(\Agl,\Sp)\big)\to \Fun^\textup{lex}\big(\Cc, \Fun^\times(\Agl,\mathscr S)\big).
    \]
    is an equivalence.
    \begin{proof}
        For any $\infty$-category $\Dd$ with finite limits, a concrete model of the stabilization is the functor $\Exc_*(\mathscr S^\text{fin}_*, \Dd)\to\Dd$ given by evaluation at $S^0$, where the source denotes the category of reduced excisive functors $\mathscr S^\text{fin}_*\to\Dd$ from the $\infty$-category of finite pointed spaces, see \cite[Definition~1.4.2.8 and Proposition~1.4.2.22]{higher-algebra}. We now observe that the canonical equivalence \[\Exc_*(\mathscr S^\text{fin}_*, \Fun(\Agl,\Dd))\simeq\Fun(\Agl, \Exc_*(\mathscr S^\text{fin},\Dd))\] from \cite[Remark~1.4.2.9]{higher-algebra} restricts to an equivalence \[\Exc_*(\mathscr S^\text{fin}_*, \Fun^\times(\Agl,\Dd))\simeq\Fun^\times(\Agl, \Exc_*(\mathscr S^\text{fin},\Dd))\] because limits in functor categories are pointwise; in other words, $\Fun^\times(\Agl,\text{--})$ commutes with stabilization. The claim now follows by applying this observation to the stabilization $\cat{R}\Omega^\infty\colon\Sp\to\mathscr S$.
    \end{proof}
\end{lemma}

To conclude our main result, it will therefore suffice to similarly exhibit $\Sp_\textup{gl}$ as the stabilization of $(\cat{$\bm\Gamma$-$\bm{E\mathcal M}$-SSet}_*^\textup{special})^\infty$. This will rely on the following global version of Segal's classical \emph{delooping theorem}:

\begin{defi}\label{defi:very-special}
    A special global $\Gamma$-space $X$ is called \emph{very special} if for every universal $H\subset\mathcal M$ the abelian monoid structure on $\pi_0(X(1^+)^H)$ defined via
    \[
        \pi_0(X(1^+)^H)^{\times 2}\cong \pi_0(X(2^+)^H)\xrightarrow{\; \pi_0(X(\mu)^H)\;} \pi_0(X(1^+)^H)
    \]
    is a group structure; here $\mu\colon 2^+\to 1^+$ denotes the map with $\mu(1)=\mu(2)=1$.
\end{defi}

\begin{thm}\label{thm:deloop}
    There exists an (explicit) functor
    \begin{equation}\label{eq:deloop}
        \textup{deloop}\colon (\cat{$\bm\Gamma$-$\bm{E\mathcal M}$-SSet}_*^\textup{special})^\infty\to\Sp_\textup{gl}
    \end{equation}
    that restricts to an equivalence $(\cat{$\bm\Gamma$-$\bm{E\mathcal M}$-SSet}_*^\textup{very special})^\infty\simeq\Sp_\textup{gl}^{\ge0}$ between the full subcategories of very special global $\Gamma$-spaces and of connective global spectra.
\end{thm}
\begin{proof}
    This is the special case $G=1$ of \cite[Theorem~3.4.21]{g-global}, translated through the equivalence from Corollary~2.2.53 of \emph{op.\ cit.}
\end{proof}

\begin{prop}\label{prop:Gamma-stab}
    The $\infty$-category $\Sp_\textup{gl}$ of global spectra is the stabilization of $(\cat{$\bm\Gamma$-$\bm{E\mathcal M}$-SSet}_*^\textup{special})^\infty$.
    \begin{proof}
        It suffices to construct an equivalence
        \[
            \Fun^\textup{lex}\big(\Dd, (\cat{$\bm\Gamma$-$\bm{E\mathcal M}$-SSet}_*^\textup{special})^\infty\big)\simeq\Fun^\textup{lex}\big(\Dd,\Sp_\textup{gl}\big)
        \]
        natural in stable $\infty$-categories $\Dd$.

        For this we first observe that any left exact functor $F\colon\Dd\to(\cat{$\bm\Gamma$-$\bm{E\mathcal M}$-SSet}_*^\textup{special})^\infty$ actually factors through the full subcategory of \emph{very special} objects: if $X\in\Dd$ is arbitrary, then $F(X)\simeq F(\Omega\Sigma X)\simeq\textbf{R}\Omega F(\Sigma X)$, and by the Eckmann-Hilton argument the monoid structure on $\pi_0(\textbf{R}\Omega F(\Sigma X)(1^+)^H)$ from Definition~\ref{defi:very-special} agrees with the group structure coming from the identification with $\pi_1(F(\Sigma X)(1^+)^H)$.

        Combining this with Theorem~\ref{thm:deloop}, we get a natural equivalence
        \[
             \Fun^\textup{lex}\big(\Dd, (\cat{$\bm\Gamma$-$\bm{E\mathcal M}$-SSet}_*^\textup{special})^\infty\big)\simeq\Fun^\textup{lex}\big(\Dd,\Sp_{\textup{gl}}^{\ge0}\big),
        \]
        so it remains to exhibit $\Sp_\text{gl}$ as the stabilization of $\Sp_\text{gl}^{\ge0}$. This follows formally from the latter being the connective part of a right complete t-structure (Proposition~\ref{prop:t-structure}), see \cite[introduction to Appendix C]{lurie-SAG} or \cite[proof of Lemma~7.2.9]{CLL_Global}.
    \end{proof}
\end{prop}

\begin{rk}
    A slightly more careful argument shows that the right adjoint of the delooping functor $(\ref{eq:deloop})$ exhibits $\Sp_\textup{gl}$ as the stabilization, see in particular \cite[proof of Theorems 3.4.21 and 3.4.32]{g-global}.
\end{rk}

We now easily get:

\begin{thm}\label{thm:main-thm}
    There exists an equivalence $\Sp_{\textup{gl}}\simeq\Fun^\times(\Agl,\Sp)$.
    \begin{proof}
        By Theorem~\ref{thm:ucom} we have an equivalence
        \[
            (\cat{$\bm\Gamma$-$\bm{E\mathcal M}$-SSet}_*^\textup{special})^\infty\simeq\Fun^\times(\Agl,\mathscr S)
        \]
        and hence in particular an equivalence of their stabilizations. The claim now follows from Proposition~\ref{prop:Gamma-stab} together with Lemma~\ref{lemma:Mackey-stab}.
    \end{proof}
\end{thm}

\appendix
\section{Unfurling for 2-categories}
This appendix is devoted to the proof of Proposition~\ref{prop:unfurl}, which we repeat for convenience:

\begin{prop*}
Let $\mathscr I$ be a strict $(2,1)$-category, and let $\mathscr I^\dagger\subset\mathscr I$ be a $2$-subcategory such that $(\nerve_\Delta(\mathscr I),\nerve_\Delta(\mathscr I),\nerve_\Delta(\mathscr I^\dagger))$ is an adequate triple. Let $\mathscr C$ be any strict $2$-category and write $\mathscr C_{(2,1)}$ for its underlying $(2,1)$-category, i.e.~the $2$-category obtained by throwing away all non-invertible $2$-cells. Moreover, let $\phi\colon\mathscr I\to\mathscr C$ be a strict $2$-functor such that for every $i\in\mathscr I^\dagger$ the functor $i_!\coloneqq\phi(i)$ admits a right adjoint $i^*$ and such that for every homotopy pullback diagram in $\nerve_\Delta(\mathscr I)$ as on the left in
\begin{equation*}
\begin{tikzcd}
A\arrow[r,"g"]\arrow[d,"j"',tail]&\twocell[dl, "\scriptstyle\sigma"{yshift=-8pt},yshift=3.5pt] B\arrow[d,"i",tail]\\
C\arrow[r,"f"'] & D
\end{tikzcd}\hskip1in
\begin{tikzcd}
\phi(A)\arrow[r,"g_!"]\twocell[dr, "\scriptstyle\sigma_\lozenge"{yshift=-8pt},yshift=3.5pt] &\phi(B)\\
\phi(C)\arrow[u,"j^*"]\arrow[r,"f_!"'] & \phi(D)\arrow[u,"i^*"']
\end{tikzcd}
\end{equation*}
with vertical arrows belonging to $\mathscr I^\dagger$, the canonical mate $\sigma_\lozenge$ of $\sigma_!$ depicted on the right is an isomorphism.

Then there is a unique functor $\Phi\colon A^{\textup{eff}}(\nerve_\Delta(\mathscr I),\nerve_\Delta(\mathscr I),\nerve_\Delta(\mathscr I^\dagger))\to\nerve_\Delta(\mathscr C_{(2,1)})$ that sends a $2$-simplex of the form
\begin{equation*}
\begin{tikzcd}[column sep=1.35em,row sep=1.875em]
&& C\arrow[ddll, bend right=40pt, shift right=5pt, "j_{02}"',""{name=j02}]\arrow[ddrr, bend left=40pt, shift left=5pt, "g_{02}", ""'{name=g02}]\arrow[dl, "j_{01}"',tail]\arrow[dr, "g_{01}"]\\
& \twocell[from=j02, "\scriptstyle\rho"{yshift=-5pt,xshift=-5pt}]B\arrow[dl, "j_{12}"',tail]\arrow[dr, "f"']\twocell[rr,"\scriptstyle\sigma"{yshift=6pt}] && D\arrow[dl, "i",tail]\arrow[dr, "g_{12}"]\twocell[from=g02,"\scriptstyle\tau"{yshift=-5pt,xshift=5pt}]\\
A && F && E
\end{tikzcd}
\end{equation*}
to the $2$-simplex
\begin{equation*}
\begin{tikzcd}[column sep=small, cramped]
& \phi(F)\arrow[dr, bend left=10pt, "g_{12!}i^*"]\\
\phi(A)\arrow[rr,"g_{02!}j_{02}^*"'{name=bot}]\arrow[ur, bend left=10pt, "f_!j_{12}^*"]\twocell[from=bot,ur] && \phi(E)
\end{tikzcd}
\quad\text{given as the pasting}\quad
\begin{tikzcd}[column sep=1.35em,row sep=1.875em]
&& \twocell[dd,"\scriptstyle\sigma_\lozenge"{xshift=10pt}]C\arrow[from=ddll, bend left=40pt, shift left=5pt, "j_{02}^*",""{name=j02}]\arrow[ddrr, bend left=40pt, shift left=5pt, "g_{02!}", ""'{name=g02}]\arrow[from=dl, "j_{01}^*"]\arrow[dr, "g_{01!}"]\\
& \twocell[from=j02, "\scriptstyle(\mskip-.3\thinmuskip\rho^*\mskip-.3\thinmuskip)^{\!{-}\!1}"{yshift=-7pt,xshift=-1pt},xshift=2pt,yshift=-2pt]B\arrow[from=dl, "j_{12}^*"]\arrow[dr, "f_!"'] && D\arrow[from=dl, "i^*"']\arrow[dr, "g_{12!}"]\twocell[from=g02,"\scriptstyle\tau_!"{yshift=-5pt,xshift=5pt}]\\
A && F && E
\end{tikzcd}
\end{equation*}
where $\rho^*$ is the total mate of $\rho_!$ and $\sigma_\lozenge$ is defined as above.
\begin{proof}
    It is clear that the above defines a map $\widehat\Phi$ from the $2$-skeleton of $A\coloneqq A^{\text{eff}}(\nerve_\Delta(\mathscr I),\nerve_\Delta(\mathscr I),\nerve_\Delta(\mathscr I^\dagger))$ to $\nerve_\Delta(\mathscr C_{(2,1)})$. As the latter is strictly $2$-coskeletal, it is then enough to show that this can be extended over $3$-simplices, i.e.~given any $3$-simplex $\sigma$ of $A$, the map $\widehat\Phi(\del\sigma)\colon\del\Delta^3\to\nerve_\Delta(\mathscr C_{(2,1)})$ extends to $\Delta^3$.

    To prove this, we begin by writing out some of the information encoded in such a $3$-simplex $\sigma$. Namely, this in particular consists of the following data:
    \begin{enumerate}
    \item A diagram
    \begin{equation*}
    \begin{tikzcd}[column sep=1.35em,row sep=1.875em]
    &&& D\arrow[dl, "j_{01}"']\arrow[dr,"g_{01}"]\\
    && C\arrow[dl, "j_{12}"']\arrow[dr, "f_{01}"{description}]\twocell[rr,"\scriptstyle\sigma"{yshift=6pt}] && E\arrow[dl, "i_{01}"{description}]\arrow[dr, "g_{12}"]\\
    & B\arrow[dl, "j_{23}"']\arrow[dr,"e"']\twocell[rr,"\scriptstyle\tau"{yshift=6pt}] && H\arrow[dl, "i_{12}"{description}]\arrow[dr,"f_{12}"{description}]\twocell[rr,"\scriptstyle\rho"{yshift=6pt}]&& F\arrow[dl,"h"]\arrow[dr, "g_{23}"]\\
    A && I && J && G
    \end{tikzcd}
    \end{equation*}
    in $\mathscr I$, where the squares are homotopy pullback squares and we have as before already pasted the natural transformations filling them.
    \item Maps $g_{02}, g_{03},g_{13}$ together with $2$-cells $\gamma_{012}$, $\gamma_{013}$, $\gamma_{023}$, $\gamma_{123}$ such that the two pastings
    \begin{equation}\label{diag:gamma-pastings}
    \begin{tikzcd}[row sep=4.5em,column sep=4.5em]
    E\arrow[r,"g_{12}"]\arrow[dr,"g_{13}"{description,name=x}]& F\arrow[d,"g_{23}"]\twocell[from=x,"\scriptstyle\gamma_{123}"{xshift=5pt,yshift=-9pt}]\\
    D\arrow[u,"g_{01}"]\arrow[r,"g_{03}"'{name=a}]\twocell[from=a,to=x,"\scriptstyle\gamma_{013}"{xshift=-11pt}]& G
    \end{tikzcd}\qquad\text{and}\qquad
    \begin{tikzcd}[row sep=4.5em,column sep=4.5em]
    E\arrow[r,"g_{12}"]& F\arrow[d,"g_{23}"]\\
    D\arrow[u,"g_{01}"]\arrow[ur,"g_{02}"{description,name=x}]\arrow[r,"g_{03}"'{name=a}]\twocell[from=a,to=x,"\scriptstyle\gamma_{023}"{xshift=12pt}]\twocell[u,from=x,"\scriptstyle\gamma_{012}"{xshift=-5pt,yshift=-9pt}]& G
    \end{tikzcd}
    \end{equation}
    agree, as well as maps $j_{02}, j_{03},j_{13}$ together with $2$-cells $\kappa_{012}$, $\kappa_{013}$, $\kappa_{023}$, $\kappa_{123}$ between the analogous composites of the $j_{ab}$'s, satisfying an analogous coherence condition.
    \item Maps $f_{02}$ and $i_{02}$ together with $2$-cells $\digamma\colon f_{02}\Rightarrow f_{12}f_{01}$ and $\iota\colon i_{02}\Rightarrow i_{12}i_{01}$.
    \end{enumerate}
    By definition of $\widehat\Phi$ and the construction of the Duskin nerve, the filling of $\widehat\Phi(\del\sigma)$ then amounts to saying that the pastings
    \begin{equation}\label{diag:Phi-hat-first}
    \begin{tikzcd}[row sep=large]
    \phi(I)\arrow[rr, "i_{12}^*"]\arrow[drr,"i_{02}^*"{description,name=c}] && \twocell[from=c,"\scriptstyle\iota^{-*}"{xshift=4pt,yshift=-5pt}]\phi(H)\arrow[d,"i_{01}^*"{description}]\arrow[rr,"f_{12!}"]&& \phi(J)\arrow[d,"h^*"]\\
    \phi(B)\arrow[u,"e_!"]\arrow[rrd,"j_{02}^*"{description,name=a}] && \phi(E)\twocell[urr, "\scriptstyle\rho_\lozenge"{xshift=5pt,yshift=-5pt}]\arrow[rr,"g_{12!}"{description}]\arrow[drr,"g_{13!}"{description,name=e}] && \phi(F)\arrow[d,"g_{23!}"]\twocell[from=e,"\scriptstyle\gamma_!"{xshift=5pt,yshift=-5pt}]\\
    \phi(A)\arrow[rr,"j_{03}^*"'{name=b}]\twocell[from=b,to=a,"\scriptstyle\kappa^{-*}"{xshift=-11pt}]\arrow[u,"j_{23}^*"] && \phi(D)\arrow[rr,"g_{03!}"'{name=f}]\twocell[to=e,from=f,"\scriptstyle\gamma_!"{xshift=-8pt}]\twocell[uull,"\scriptstyle(\sigma\odot\tau)_\lozenge"{yshift=-8pt},yshift=4pt]\arrow[u,"g_{01!}"{description}] && \phi(G)
    \end{tikzcd}
    \end{equation}
    and
    \begin{equation}\label{diag:Phi-hat-second}
    \begin{tikzcd}[row sep=large]
    \phi(I)\arrow[rr, "i_{12}^*"] && \phi(H)\arrow[rr,"f_{12!}"]&& \phi(J)\arrow[d,"h^*"]\\
    \phi(B)\arrow[u,"e_!"]\arrow[rr,"j_{12}^*"{description}] && \phi(C)\twocell[llu,"\scriptstyle\tau_\lozenge"{xshift=-5pt,yshift=-5pt}]\arrow[urr,"f_{02!}"{description,name=x}]\twocell[from=x,u,"\scriptstyle\digamma_!"{xshift=5pt,yshift=4pt}]\arrow[u,"f_{01!}"{description}]\arrow[d,"j_{01}^*"{description}] && \phi(F)\arrow[d,"g_{23!}"]\\
    \phi(A)\arrow[urr,"j_{13}^*"{description,name=b}]\twocell[from=b,u,"\scriptstyle\kappa^{-*}"{xshift=-5pt,yshift=-5pt}]\arrow[rr,"j_{03}^*"'{name=c}]\twocell[from=c,to=b,"\scriptstyle\kappa^{-*}"{xshift=11pt}]\arrow[u,"j_{23}^*"] && \phi(D)\twocell[uurr,"\scriptstyle(\rho\odot\sigma)_\lozenge"{yshift=-8.5pt},yshift=4pt]\arrow[rr,"g_{03!}"'{name=a}]\arrow[urr,"g_{02!}"{description,name=y}]\twocell[from=a,to=y,"\scriptstyle\gamma_!"{xshift=8pt}] && \phi(G)
    \end{tikzcd}
    \end{equation}
    agree, where $\sigma\odot\tau$ denotes the pasting
    \begin{equation*}
    \begin{tikzcd}[row sep=large,column sep=large]
    D\arrow[d,"g_{01}"']\arrow[r,"j_{01}"{description}]\arrow[rr, bend left=25pt,yshift=7.5pt,"j_{02}"{name=a}] & C\twocell[dl,"\scriptstyle\sigma"{xshift=5pt,yshift=-5pt}]\arrow[d,"f_{01}"{description}]\twocell[from=a,"\scriptstyle\kappa"{xshift=8pt},yshift=-2pt]\arrow[r,"j_{12}"{description}] & \twocell[dl,"\scriptstyle\tau"{xshift=5pt,yshift=-5pt}]B\arrow[d,"e"]\\
    E\arrow[r,"i_{01}"{description}]\arrow[rr,bend right=25pt,yshift=-7.5pt, "i_{02}"'{name=b}] & H\twocell[to=b,"\scriptstyle\iota^{-1}"{xshift=11pt,yshift=1pt}]\arrow[r, "i_{12}"{description}] & I\rlap{,}
    \end{tikzcd}
    \end{equation*}
    $\rho\odot\sigma$ is defined analogously, and we have for readability omitted the indices of the transformations $\gamma$ and $\kappa$, and abbreviated $\kappa^{-*}=(\kappa^*)^{-1}$ etc.

    To prove the equality of (\ref{diag:Phi-hat-first}) and (\ref{diag:Phi-hat-second}), we first observe that applying $2$-functoriality to (\ref{diag:gamma-pastings}) shows that the pastings
    \begin{equation}\label{diag:gamma-shriek-pastings}
    \begin{tikzcd}[row sep=4.5em,column sep=4.5em]
    \phi(E)\arrow[r,"g_{12!}"]\arrow[dr,"g_{13!}"{description,name=x}]& \phi(F)\arrow[d,"g_{23!}"]\twocell[from=x,"\scriptstyle\gamma_{!}"{xshift=5pt,yshift=-9pt}]\\
    \phi(D)\arrow[u,"g_{01!}"]\arrow[r,"g_{03!}"'{name=a}]\twocell[from=a,to=x,"\scriptstyle\gamma_{!}"{xshift=-8pt}]& \phi(G)
    \end{tikzcd}\qquad\text{and}\qquad
    \begin{tikzcd}[row sep=4.5em,column sep=4.5em]
    \phi(E)\arrow[r,"g_{12!}"]& \phi(F)\arrow[d,"g_{23!}"]\\
    \phi(D)\arrow[u,"g_{01!}"]\arrow[ur,"g_{02!}"{description,name=x}]\arrow[r,"g_{03!}"'{name=a}]\twocell[from=a,to=x,"\scriptstyle\gamma_{!}"{xshift=8pt}]\twocell[u,from=x,"\scriptstyle\gamma_{!}"{xshift=-5pt,yshift=-9pt}]& \phi(G)
    \end{tikzcd}
    \end{equation}
    agree. Arguing likewise for $\kappa$ and then appealing to the compatibility of mates with pastings moreover shows that the two pastings
    \begin{equation}\label{diag:kappa-star-pastings}
    \null\quad\hskip-2pt
    \begin{tikzcd}[row sep=4.5em,column sep=4.5em]
    \phi(B)\arrow[r,"j_{12}^*"]\arrow[dr,"j_{13}^*"{description,name=x}]& \phi(C)\arrow[d,"j_{23}^*"]\twocell[from=x,"\scriptstyle\kappa^{-*}"{xshift=5pt,yshift=-9pt}]\\
    \phi(A)\arrow[u,"j_{01}^*"]\arrow[r,"j_{03}^*"'{name=a}]\twocell[from=a,to=x,"\scriptstyle\kappa^{-*}"{xshift=-10pt}]& \phi(D)
    \end{tikzcd}\qquad\text{and}\qquad
    \begin{tikzcd}[row sep=4.5em,column sep=4.5em]
    \phi(B)\arrow[r,"j_{12}^*"]& \phi(C)\arrow[d,"j_{23}^*"]\\
    \phi(A)\arrow[u,"j_{01}^*"]\arrow[ur,"j_{02}^*"{description,name=x}]\arrow[r,"g_{03!}"'{name=a}]\twocell[from=a,to=x,"\scriptstyle\kappa^{-*}"{xshift=10pt}]\twocell[u,from=x,"\scriptstyle\kappa^{-*}"{xshift=-5pt,yshift=-9pt}]& \phi(D)
    \end{tikzcd}
    \end{equation}
    agree. Finally, again using $2$-functoriality and the compatibility of mates with pastings shows that $(\sigma\odot\tau)_\lozenge$ agrees with the pasting
    \begin{equation}\label{diag:lozenge-pasting}
    \begin{tikzcd}[column sep=large,row sep=large]
    \phi(I)\arrow[r,"i_{12}^*"{description}]\arrow[rr, bend left=25pt,yshift=7.5pt,"i_{02}^*"{name=a}] & \phi(H)\twocell[to=a,"\scriptstyle\iota^*"{xshift=8pt},yshift=-2pt]\arrow[r,"i_{01}^*"{description}] & \phi(E)\\
    \phi(B)\arrow[u,"e_!"]\arrow[r,"j_{12}^*"{description}]\arrow[rr,bend right=25pt,yshift=-7.5pt,"j_{02}^*"'{name=b}] & \phi(C)\twocell[ul,"\scriptstyle\tau_\lozenge"{xshift=5pt,yshift=5pt}]\arrow[u,"f_{01!}"{description}]\arrow[r,"j_{01}^*"{description}]\twocell[from=b,"\scriptstyle\kappa^{-*}"{xshift=11pt},yshift=2pt] & \phi(D)\arrow[u,"g_{01!}"']\twocell[ul,"\scriptstyle\sigma_\lozenge"{xshift=5pt,yshift=5pt}]
    \end{tikzcd}
    \end{equation}
    (where we have additionally rotated the diagram by $\pi$ radians for the sake of the argument below), and analogously for $(\rho\odot\sigma)_\lozenge$.

    Now plugging (\ref{diag:lozenge-pasting}) and (\ref{diag:kappa-star-pastings}) into the left hand portion of (\ref{diag:Phi-hat-first}) shows that the whole diagram (\ref{diag:Phi-hat-first}) agrees with the pasting
    \begin{equation}\label{diag:Phi-hat-first-massaged}
    \begin{tikzcd}[row sep=large]
    \phi(I)\arrow[rr, "i_{12}^*"] && \phi(H)\arrow[d,"i_{01}^*"{description}]\arrow[rr,"f_{12!}"]&& \phi(J)\arrow[d,"h^*"]\\
    \phi(B)\arrow[u,"e_!"]\arrow[r,"j_{12}^*"{description,name=a}] & \phi(C)\twocell[ul,"\scriptstyle\tau_\lozenge"{xshift=8pt}]\arrow[ur,"f_{01!}"{description,name=s2}]\arrow[dr,"j_{01}^*"{description,name=s1}]\twocell[from=s1,to=s2,"\scriptstyle\sigma_\lozenge"{xshift=8pt,yshift=-1.5pt},xshift=-4pt]
    & \phi(E)\twocell[urr, "\scriptstyle\rho_\lozenge"{xshift=5pt,yshift=-5pt}]\arrow[rr,"g_{12!}"{description}]\arrow[drr,"g_{13!}"{description,name=e}] && \phi(F)\arrow[d,"g_{23!}"]\twocell[from=e,"\scriptstyle\gamma_!"{xshift=5pt,yshift=-5pt}]\\
    \phi(A)\arrow[ur,"j_{13}^*"{description,name=x}]\twocell[from=x,u,"\scriptstyle\kappa^{-*}"{xshift=-8pt,yshift=-6pt},xshift=4pt,yshift=4pt]\arrow[rr,"j_{03}^*"'{name=b}]\twocell[from=b,ur,"\scriptstyle\kappa^{-*}"{xshift=-11pt}]\arrow[u,"j_{23}^*"] && \phi(D)\arrow[rr,"g_{03!}"'{name=f}]\twocell[to=e,from=f,"\scriptstyle\gamma_!"{xshift=-8pt}]\arrow[u,"g_{01!}"{description}] && \phi(G)\rlap{.}
    \end{tikzcd}
    \end{equation}
    Similarly using (\ref{diag:gamma-shriek-pastings}) together with the analogue of (\ref{diag:lozenge-pasting}) for $(\rho\odot\sigma)_\lozenge$ left implicit above, we see that (\ref{diag:Phi-hat-second}) agrees with the pasting
    \begin{equation*}
    \begin{tikzcd}[row sep=large]
    \phi(I)\arrow[rr, "i_{12}^*"] && \phi(H)\arrow[dr,"i_{01}^*"{description,name=y}]\arrow[rr,"f_{12!}"]&& \phi(J)\arrow[d,"h^*"]\\
    \phi(B)\arrow[u,"e_!"]\arrow[rr,"j_{12}^*"{description}] && \phi(C)\twocell[llu,"\scriptstyle\tau_\lozenge"{xshift=-5pt,yshift=-5pt}]\arrow[u,"f_{01!}"{description}]\arrow[d,"j_{01}^*"{description}] &\phi(E)\arrow[r,"g_{12!}"{description}]\twocell[ur,"\scriptstyle\rho_\lozenge"{xshift=-6pt,yshift=5pt}]\arrow[dr,"g_{13!}"{description,name=z}]\twocell[from=z,r,"\scriptstyle\gamma_!"{xshift=6pt,yshift=-6pt},xshift=-3pt]& \phi(F)\arrow[d,"g_{23!}"]\\
    \phi(A)\arrow[urr,"j_{13}^*"{description,name=b}]\twocell[from=b,u,"\scriptstyle\kappa^{-*}"{xshift=-5pt,yshift=-5pt}]\arrow[rr,"j_{03}^*"'{name=c}]\twocell[from=c,to=b,"\scriptstyle\kappa^{-*}"{xshift=11pt}]\arrow[u,"j_{23}^*"] && \phi(D)\arrow[ur,"g_{01!}"{description,name=x}]\twocell[from=x,to=y,"\scriptstyle\sigma_\lozenge"{xshift=8pt,yshift=-1.5pt},xshift=-4pt]\arrow[rr,"g_{03!}"'{name=a}]\twocell[from=a,ur,"\scriptstyle\gamma_!"{xshift=8pt}] && \phi(G)\rlap{.}
    \end{tikzcd}
    \end{equation*}
    However, this is in turn exactly the same as (\ref{diag:Phi-hat-first-massaged}) except for the way we have embedded it into the plane, which completes the proof of the proposition.
\end{proof}
\end{prop*}

\frenchspacing
\bibliographystyle{amsalpha}
\bibliography{literature.bib}
\end{document}